\documentclass[11pt]{amsart}

\calclayout

\usepackage[utf8]{inputenc}
\usepackage{fontenc}
\usepackage{amsfonts}
\usepackage{amssymb}
\usepackage{amsmath,pict2e}
\usepackage{amsthm}\usepackage{mathtools}
\usepackage{enumerate}
\usepackage{hyperref}\usepackage{enumitem}
\usepackage{mathrsfs}
\usepackage{tikz}\usetikzlibrary{calc}
\usepackage{marginnote}
\usepackage{xcolor,enumitem}
\usepackage{soul}
\usetikzlibrary{positioning,calc}
\usepackage[all,cmtip]{xy}
 
\usepackage{tikz-cd}
\usetikzlibrary{matrix,arrows}
\newlength{\myarrowsize}
 \newenvironment{diagram}[2]{%
\begin{equation*}%
\begin{tikzpicture}[>=cmto,baseline=(current bounding box.center),%
	to/.style={-cmto,font=\scriptsize,cap=round},%
	into/.style={cmhook->,font=\scriptsize,cap=round},%
	onto/.style={-cmonto,font=\scriptsize,cap=round},%
	math/.style={matrix of math nodes, row sep=#2, column sep=#1,%
		text height=1.5ex, text depth=0.25ex}]%
}{%
\end{tikzpicture}%
\end{equation*}%
\ignorespacesafterend%
}
\pgfarrowsdeclare{cmto}{cmto}{
	\pgfsetdash{}{0pt} 
	\pgfsetbeveljoin 
	\pgfsetroundcap 
	\setlength{\myarrowsize}{0.6pt}
	\addtolength{\myarrowsize}{.5\pgflinewidth}
	\pgfarrowsleftextend{-4\myarrowsize-.5\pgflinewidth} 
	\pgfarrowsrightextend{.8\pgflinewidth}
}{
	\setlength{\myarrowsize}{0.6pt} 
  	\addtolength{\myarrowsize}{.5\pgflinewidth}  
	\pgfsetlinewidth{0.5\pgflinewidth}
	\pgfsetroundjoin
	\pgfpathmoveto{\pgfpoint{1.5\pgflinewidth}{0}}
	\pgfpatharc{-109}{-170}{4\myarrowsize}
	\pgfpatharc{10}{189}{0.58\pgflinewidth and 0.2\pgflinewidth}
	\pgfpatharc{-170}{-115}{4\myarrowsize+\pgflinewidth}
	\pgfpathclose
	\pgfusepathqfillstroke
	\pgfpathmoveto{\pgfpoint{1.5\pgflinewidth}{0}}
	\pgfpatharc{109}{170}{4\myarrowsize}
	\pgfpatharc{-10}{-189}{0.58\pgflinewidth and 0.2\pgflinewidth}
	\pgfpatharc{170}{115}{4\myarrowsize+\pgflinewidth}
	\pgfpathclose
	\pgfusepathqfillstroke
	\pgfsetlinewidth{2\pgflinewidth}
}

\pgfarrowsdeclare{cmonto}{cmonto}{
	\pgfsetdash{}{0pt} 
	\pgfsetbeveljoin 
	\pgfsetroundcap 
	\setlength{\myarrowsize}{0.6pt}
	\addtolength{\myarrowsize}{.5\pgflinewidth}
	\pgfarrowsleftextend{-4\myarrowsize-.5\pgflinewidth} 
	\pgfarrowsrightextend{.8\pgflinewidth}
}{
	\setlength{\myarrowsize}{0.6pt} 
  	\addtolength{\myarrowsize}{.5\pgflinewidth}  
	\pgfsetlinewidth{0.5\pgflinewidth}
	\pgfsetroundjoin
	\pgfpathmoveto{\pgfpoint{1.5\pgflinewidth}{0}}
	\pgfpatharc{-109}{-170}{4\myarrowsize}
	\pgfpatharc{10}{189}{0.58\pgflinewidth and 0.2\pgflinewidth}
	\pgfpatharc{-170}{-115}{4\myarrowsize+\pgflinewidth}
	\pgfpathclose
	\pgfusepathqfillstroke
	\pgfpathmoveto{\pgfpoint{1.5\pgflinewidth}{0}}
	\pgfpatharc{109}{170}{4\myarrowsize}
	\pgfpatharc{-10}{-189}{0.58\pgflinewidth and 0.2\pgflinewidth}
	\pgfpatharc{170}{115}{4\myarrowsize+\pgflinewidth}
	\pgfpathclose
	\pgfusepathqfillstroke
	\pgfpathmoveto{\pgfpoint{1.5\pgflinewidth-0.3em}{0}}
	\pgfpatharc{-109}{-170}{4\myarrowsize}
	\pgfpatharc{10}{189}{0.58\pgflinewidth and 0.2\pgflinewidth}
	\pgfpatharc{-170}{-115}{4\myarrowsize+\pgflinewidth}
	\pgfpathclose
	\pgfusepathqfillstroke
	\pgfpathmoveto{\pgfpoint{1.5\pgflinewidth-0.3em}{0}}
	\pgfpatharc{109}{170}{4\myarrowsize}
	\pgfpatharc{-10}{-189}{0.58\pgflinewidth and 0.2\pgflinewidth}
	\pgfpatharc{170}{115}{4\myarrowsize+\pgflinewidth}
	\pgfpathclose
	\pgfusepathqfillstroke
	\pgfsetlinewidth{2\pgflinewidth}
}

\pgfarrowsdeclare{cmhook}{cmhook}{
	\pgfsetdash{}{0pt} 
	\pgfsetbeveljoin 
	\pgfsetroundcap 
	\setlength{\myarrowsize}{0.6pt}
	\addtolength{\myarrowsize}{.5\pgflinewidth}
	\pgfarrowsleftextend{-4\myarrowsize-.5\pgflinewidth} 
	\pgfarrowsrightextend{.8\pgflinewidth}
}{
	\setlength{\myarrowsize}{0.6pt} 
  	\addtolength{\myarrowsize}{.5\pgflinewidth}  
 	\pgfsetdash{}{0pt}
	\pgfsetroundcap
	\pgfpathmoveto{\pgfqpoint{0pt}{-4.667\pgflinewidth}}
	\pgfpathcurveto
    {\pgfqpoint{4\pgflinewidth}{-4.667\pgflinewidth}}
    {\pgfqpoint{4\pgflinewidth}{0pt}}
    {\pgfpointorigin}
	\pgfusepathqstroke
}

\newcommand{\R}{\mathbb{R}}
\newcommand{\Z}{\mathbb{Z}}
\newcommand{\Zp}{\overline{\mathbb{Z}}_{+}}

\newcommand{\N}{\mathbb{N}}
\newcommand{\C}{\mathbb{C}}

\newcommand{\cE}{\mathcal{E}}

\newcommand{\eps}{\varepsilon}

\newcommand{\cstu}{\mathrm{C}^*_u}

\newtheorem*{rigprob*}{Rigidity Problem for uniform Roe Algebras}
\newtheorem*{rigprobcorona*}{Rigidity Problem for uniform Roe Coronas}

\newcommand{\cstar}{$\mathrm{C}^*$}

\newcommand{\cF}{\mathcal{F}}

\newcommand{\cB}{\mathcal{B}}
\newcommand{\cK}{\mathcal{K}}

\newcommand{\ICOD}{\mathrm{ICOD}}
\newcommand{\ISOD}{\mathrm{ISOD}}
\newcommand{\SUPPEXP}{\mathrm{SUPPEXP}}

\newtheorem{theorem}{Theorem}[section]
\newtheorem*{theorem*}{Theorem}
\newtheorem{proposition}[theorem]{Proposition}

\newtheorem*{proposition*}{Proposition}
\newtheorem{lemma}[theorem]{Lemma}
\newtheorem*{lemma*}{Lemma}
\newtheorem{corollary}[theorem]{Corollary}
\newtheorem*{corollary*}{Corollar}
\newtheorem{fact}[theorem]{Fact}
\newtheorem*{fact*}{Fact}
\theoremstyle{definition}
\newtheorem{definition}[theorem]{Definition}
\newtheorem*{definition*}{Definition}
\newtheorem{claim}[theorem]{Claim}
\newtheorem*{claim*}{Claim}

\newtheorem*{acknowledgements}{Acknowledgements}
\newtheorem*{conjecture*}{Conjecture}

\newtheorem{theoremi}{Theorem}

\theoremstyle{remark}
\newtheorem{example}[theorem]{Example}
\newtheorem*{example*}{Example}
\newtheorem{remark}[theorem]{Remark}
\newtheorem*{remark*}{Remark}

\newtheorem*{note*}{Note}
\newtheorem*{question*}{Question}

\newcommand{\norm}[1]{\left\lVert #1 \right\rVert}

\newcommand{\brak}[1]{\left < #1 \right >}

\newcommand{\Q}{\mathbb{Q}}

\newcommand{\supp}{\mathrm{supp}}

\makeatletter
\newcommand{\bigcomp}{%
  \DOTSB
  \mathop{\vphantom{\sum}\mathpalette\bigcomp@\relax}%
  \slimits@
}
\newcommand{\bigcomp@}[2]{%
  \begingroup\m@th
  \sbox\z@{$#1\sum$}%
  \setlength{\unitlength}{0.9\dimexpr\ht\z@+\dp\z@}%
  \vcenter{\hbox{%
    \begin{picture}(1,1)
    \bigcomp@linethickness{#1}
    \put(0.5,0.5){\circle{1}}
    \end{picture}%
  }}%
  \endgroup
}
\newcommand{\bigcomp@linethickness}[1]{%
  \linethickness{%
      \ifx#1\displaystyle 2\fontdimen8\textfont\else
      \ifx#1\textstyle 1.65\fontdimen8\textfont\else
      \ifx#1\scriptstyle 1.65\fontdimen8\scriptfont\else
      1.65\fontdimen8\scriptscriptfont\fi\fi\fi 3
  }%
}
\makeatother

\newcommand{\cG}{\mathcal G}

\newcounter{my_enumerate_counter}
\newcommand{\pushcounter}{\setcounter{my_enumerate_counter}{\value{enumi}}}
\newcommand{\popcounter}{\setcounter{enumi}{\value{my_enumerate_counter}}}

\usepackage{enumitem}

\title{Support expansion \cstar-algebras}

\author[B. M. Braga]{Bruno M. Braga}

\address[B. M. Braga]{PUC-Rio, Rua Marquês de São Vicente 225, Rio de Janeiro, RJ, Brazil.}
\email{demendoncabraga@gmail.com}

\author[J. Eisner]{Joseph Eisner}
\address[J. Eisner]{ University of Virginia, $141$ Cabell Drive, Kerchof Hall, P.O. Box $400137$, Charlottesville, VA $22904$-$4137$,  USA} \email{ je5pd@virginia.edu}

\author[D. Sherman]{David Sherman}
\address[D. Sherman]{University of Virginia, $141$ Cabell Drive, Kerchof Hall, P.O. Box $400137$, Charlottesville, VA $22904$-$4137$,  USA } \email{  dsherman@virginia.edu}

\date{\today}
\begin{document}
\maketitle

 \begin{abstract}
We consider operators on $L^2$ spaces that expand the support of vectors in a manner controlled by some constraint function.  The primary objects of study are \cstar-algebras that arise from suitable families of constraints, which we call \textit{support expansion \cstar-algebras}.  In the discrete setting, support expansion \cstar-algebras are classical uniform Roe algebras, and the continuous version featured here provides examples of ``measurable" or ``quantum" uniform Roe algebras as developed in a companion paper.  We find that in contrast to the discrete setting, the poset of support expansion \cstar-algebras inside $\cB(L^2(\R))$ is extremely rich, with uncountable ascending chains, descending chains, and antichains.
 \end{abstract}

\setcounter{tocdepth}{1}
\tableofcontents

\section{Introduction}\label{SectionIntro}
In this paper we initiate the study of operators on Hilbert spaces of $L^2$ functions which expand the support of vectors in a manner controlled by a constraint function.  Our primary focus is on \cstar-algebras that arise from families of constraints, which we call \emph{support expansion \cstar-algebras}.  These are new objects with a simple definition.  In the discrete setting, support expansion \cstar-algebras are uniform Roe algebras, while the continuous version featured here provides examples of ``measurable" or ``quantum" uniform Roe algebras, as introduced in our companion paper \cite{BragaEisnerShermanQuantum}.

Faced with a class of \cstar-algebras, an upstanding mathematician will try to determine what an algebra remembers about the data used to generate it.  For uniform Roe algebras arising from metric spaces with bounded geometry, we now know that stable isomorphism of algebras corresponds exactly to coarse equivalence of metric spaces \cite[Theorem 1.4]{BaudierBragaFarahKhukhroVignatiWillett2021uRaRig}.  Support expansion \cstar-algebras analogously encode rates of growth of the constraint functions.  In the discrete setting of $\cB(\ell^2)$, the variety of possible growth rates is poor: the poset of support expansion \cstar-algebras, ordered by inclusion, collapses to four linearly ordered elements.  But inside $\cB(L^2(\R))$, this poset is extremely rich.  We are able to describe its top and bottom explicitly, while the middle is a jungle containing uncountable ascending chains, descending chains, and antichains.


In the rest of this introduction we make the preceding paragraphs more precise, although supporting details are naturally postponed to later sections of the text.


 Given a measure space $(X,\Sigma,\mu)$, $L^2(X,\mu)$ denotes the Hilbert space of all equivalence classes of square-integrable $\mu$-measurable functions $X\to \C$, and $\cB(L^2(X,\mu))$ denotes the space of all bounded linear operators on $L^2(X,\mu)$. The \emph{support} of $\xi \in L^2(X, \mu)$ is \[ \text{supp}(\xi) = \{x \in X|\: \xi(x) \neq 0\} \subseteq X,\]which is well-defined up to a null set.  


Given an increasing\footnote{In this paper a function $f$ is ``increasing" when $x < y \Rightarrow f(x) \leq f(y)$.  We say ``stricly increasing" when $x < y \Rightarrow f(x) < f(y)$.} function $f\colon [0,\infty]\to [0,\infty]$, which we think of as a constraint, we say that $a\in \cB(L^2(X,\mu))$ has \emph{support expansion  controlled by $f$} if
\[\mu (\supp(a\xi)) \leq f(\mu(\supp(\xi)) ), \quad \forall \xi\in L^2(X,\mu).\]
Arguably, the most natural constraints are those of the form $f(x)=\lambda x$. 

Now let $\cF$ be a nonempty family of increasing functions $[0,\infty]\to [0,\infty]$ which is  closed under addition and composition, and define
\[B_{\cF}=\Big\{a\in \cB(L^2(X,\mu) : a\text{ and }a^*\text{ are controlled by some }f\in \cF\Big\}.\]
The conditions on $\cF$ imply that $B_{\cF}$ is a $*$-subalgebra of $\cB(L^2(X,\mu))$. Its norm closure
\[C_{\cF}=\overline{B_{\cF}}^{\|\cdot\|}\subseteq\cB(L^2(X,\mu))\]
is then a \cstar-algebra, which we call a \emph{support expansion \cstar-algebra}.

Despite the very simple nature of the \cstar-algebras $C_\cF$ above, they seem to be new.  To the best of our knowledge, the only nontrivial example in the literature is the case where $\mu$ is counting measure on $\N$ and $\cF$ is the family of all  linear functions. In that case, an operator $a\in \cB(\ell^2(\N))$ is in $B_\cF$ if and only if there is $\lambda\in (0,\infty)$ such that its standard matrix representation $a=[a_{n,m}]_{n,m\in\N}$ satisfies 
\[ \left|\{n\in \N : a_{n,m}\neq 0 \}\right|,\left|\{n\in \N : a_{m,n}\neq 0 \}\right| \leq \lambda, \qquad \forall m \in \N,\]
i.e., $a=[a_{n,m}]_{n,m\in\N}$ has at most $\lambda$ nonzero entries in each column and row. We call such an operator \emph{uniformly row and column finite}, and we let $C_{\mathrm{RC}}$ denote  the norm closure of all such operators, i.e., $C_{\mathrm{RC}}=C_{\cF}$.

The \cstar-algebra $C_{\mathrm{RC}}$  was recently studied by V. Manuilov in \cite{Manuilov2019}; the reader familiar with uniform Roe algebras will notice that $C_{\mathrm{RC}}$ is the uniform Roe algebra of the \emph{largest uniformly locally finite coarse structure on $\N$}. In particular, $C_{\mathrm{RC}}$ contains isomorphic copies of the uniform Roe algebras of all uniformly locally finite metric spaces. Although this link to uniform Roe algebra theory serves as motivation for the study of more general \cstar-algebras of the form $C_{\cF}$, the technicalities of uniform Roe algebras do not play  an important role in this paper. For this reason,   we refer the reader to \cite{RoeBook} and \cite{BragaFarahVignati2020AnnInstFour,BaudierBragaFarahKhukhroVignatiWillett2021uRaRig}
for more details on coarse spaces and their uniform Roe algebras  (see also Remark \ref{RemarkURA} below). We also point out that our companion paper \cite{BragaEisnerShermanQuantum} works out the basic theory of support expansion \cstar-algebras as examples of \emph{measurable} uniform Roe algebras or, more generally, \emph{quantum} uniform Roe algebras.


We now  briefly describe   our main results. Firstly, in Section \ref{SectionSuppExpDiscrete}, we show that the procedure above of constructing support expansion \cstar-algebras is very limited when   $\mu$ is the counting measure on $\N$. Let $\cK(\ell^2(\N))$ denote the ideal of compact operators on $\ell^2(\N)$.

\begin{theoremi}[Proved below as Theorem \ref{thmDiscContrExpC*Algs}]
The set of support expansion \cstar-subalgebras of $\cB(\ell^2(\N))$ is \[\Big\{ \{0\}, \mathcal{K}(\ell^2(\N)), C_{\mathrm{RC}}, \mathcal{B}(\ell^2(\N)) \Big\}.\]\label{thmDiscContrExpC*AlgsINTRO}
\end{theoremi}

Inside $\cB(L^2(\R))$, the situation is drastically different, and this is the subject of the remainder of this paper. To fix notation, we introduce the following:

\begin{definition}\label{DefintionPIntro}
Endow $\R$ with the Lebesgue measure. We denote by $\mathbb P$   the poset of all support expansion \cstar-subalgebras of $\mathcal{B}(L^2(\R))$, with order given by inclusion.  
\end{definition}

In the definition of $\mathbb P$ above, we are allowed to use \textit{all} increasing functions $[0,\infty]\to [0,\infty]$. Reducing our problems to smaller, more tractable classes of functions is part of the focus of  Sections \ref{SectionSuppExpCont} and  \ref{SectionAlgebrasContExp}, culminating 	 with the following.   Here we denote the set of all increasing concave down functions $[0, \infty]\to [0,\infty]$ taking 0 to 0 by the acronym   $\ICOD$.

\begin{theorem}[Proved as Theorem \ref{ThmpropSinglyGeneratedC*AlgCountablyGeneratedByISOD} below]
Any element of  $\mathbb P$ can be written as $C_{\cF}$ for some nonempty family $\cF \subseteq \ICOD$ that is closed under addition and composition.
\end{theorem}

Restricting to $\ICOD$ functions helps us analyze $\mathbb P$.  In order to state our next result describing   the ``tail'' of $\mathbb P$, we need some extra notation:

\begin{itemize}
\item $\ICOD_0=\{f\in \ICOD : \lim_{t\to 0}f(t)=0\}$,
\item  $\ICOD_{\mathrm{bdd}}=\{f\in \ICOD  : f\text{ is bounded}\}$,

\item $\ICOD_{0\cap \mathrm{bdd}}=\ICOD_{0 }\cap \ICOD_{ \mathrm{bdd}}$, and
\item $\ICOD_{<\infty}=\{f\in \ICOD : f(x)<\infty \text{ for all }x<\infty\}$.
\end{itemize}

\begin{theoremi} [Proved as Theorem \ref{thmTopPOSETContainments} below]
Within $\mathbb{P}$, all the following inclusions are strict and have no intermediate elements. \begin{diagram}{1.3em}{0em}
 \matrix[math] (m) { \ & C_{\ICOD_0}  & \  & \   \\ C_{\ICOD_{0 \cap \mathrm{bdd}}} & \ & C_{\ICOD_{<\infty}} &  C_{\ICOD} =\cB(L^2(\R))\\ \ &  C_{\ICOD_{\mathrm{bdd}}} & \ & \  \\ }; 
 \path (m-2-1) edge[draw=none] node [sloped, auto=false, allow upside down] {$\subsetneq$} (m-1-2);
 \path (m-1-2) edge[draw=none] node [sloped, auto=false, allow upside down] {$\subsetneq$} (m-2-3);
 \path (m-2-1) edge[draw=none] node [sloped, auto=false, allow upside down] {$\subsetneq$} (m-3-2);
 \path (m-3-2) edge[draw=none] node [sloped, auto=false, allow upside down] {$\subsetneq$} (m-2-3);
 \path (m-2-3) edge[draw=none] node [sloped, auto=false, allow upside down] {$\subsetneq$} (m-2-4);
 \end{diagram}
 Moreover, $C_{\ICOD_0}$ and  $C_{\ICOD_{\mathrm{bdd}}}$ are the only elements between $C_{\ICOD_{0\cap \mathrm{bdd}}}$ and $C_{\ICOD_{<\infty}}$. 
 \label{thmTopPOSETContainmentsIntro}
\end{theoremi}

In Remark \ref{fullposet} we ``complete" the diagram in Theorem \ref{thmTopPOSETContainmentsIntro} and discuss the way support expansion \cstar-algebras encode the growth of constraint functions at 0 and $\infty$.

Section \ref{SectionCharFunct} characterizes an inclusion $C_{\cF}\subseteq C_{\cG}$ in terms of simple comparisons between functions (see Theorem  \ref{thmCharacterizationOfMultiFGenC*AlgsFromFunctionProperties} and Corollary \ref{corCharacterizationOfSingleFGenC*AlgsFromFunctionProperties}). 
These characterizations are our main tools for a further analysis of $\mathbb P$ in Section \ref{SectionLargePoset}, leading to the following main result.  It illustrates   how different the poset $\mathbb P$ is from its discrete version   (cf.\ Theorem \ref{thmDiscContrExpC*AlgsINTRO}): 
 
\begin{theoremi}[Proved as Theorem \ref{ThmMAIN} below]
${}$
\begin{enumerate}
\item $\mathbb P$ has uncountable ascending chains,
\item $\mathbb P$ has uncountable descending  chains, and
\item $\mathbb P$ has uncountable antichains.
\end{enumerate}\label{ThmMAINIntro}
\end{theoremi}

In this paper $\N$ denotes the set of nonnegative integers, and  $\overline{\N}=\N \cup\{\infty\}$. Many results here were originally obtained in the second-named author's 2021 PhD dissertation at the University of Virginia \cite{EisnerDissertation}.

\section{A few preparatory lemmas about support}
Let $(X, \mu)$ be a measure space.  All set equations and containments in this section are understood to hold off a null set.

\begin{lemma} \label{musupportprops}
The function $\mu(\supp(\cdot)): L^2(X, \mu) \to [0,\infty]$ enjoys the following properties.
\begin{enumerate}
    \item (lower semicontinuity) If $\xi_n \to \xi$ in $L^2(X, \mu)$, then \[\mu(\supp(\xi)) \leq \liminf_{n\to \infty} \mu(\supp(\xi_n)).\]    
    \item (subbadditivity) If $\xi, \eta \in L^2(X,\mu)$, then \[\mu(\supp(\xi + \eta)) \leq \mu(\supp(\xi)) + \mu( \supp(\eta)).\]
\end{enumerate}
\end{lemma}

The first is an exercise in measure theory.  
The second follows immediately from the containment
\[\supp(\xi + \eta) \subseteq \supp(\xi) \cup \supp(\eta),\]
which just says that the sum of two functions can only be nonzero where at least one of them is nonzero.  Note that this containment will be proper if there is cancellation between $\xi$ and $\eta$ on a nonnull set.  One can always guarantee equality (cancellation only on a null set) by scaling one of the functions.  

\begin{lemma}\label{lemCyclicProjections}
Given  $\xi_1, \xi_2 \in L^2(X,\mu)$, we have that \[ \supp(\xi_1 + \lambda \xi_2) = \supp(\xi_1) \cup \supp(\xi_2) \]  for all but perhaps countably many $\lambda \in \C$.
\end{lemma}

\begin{proof}
Let $E = \supp(\xi_1) \cup \supp(\xi_2)$, and for each $\lambda \in \C$ let $E_\lambda = E \setminus \supp(\xi_1 + \lambda \xi_2)$. We first show that the $E_\lambda$ are pairwise disjoint. Indeed, fix $\lambda\neq \lambda'$  in $\C$ and notice that both $\xi_1 + \lambda \xi_2$ and $\xi_1 + \lambda^{'} \xi_2$ are 0 on $E_\lambda \cap E_{\lambda'}$.  As $\lambda\neq \lambda'$, this implies that $\xi_1$ and $\xi_2$ are each 0 on $E_\lambda \cap E_{\lambda'}$.  Therefore $E_\lambda \cap E_{\lambda'}$ is disjoint from $\supp(\xi_1) \cup \supp(\xi_2)=E$, of which it is a subset, and must be empty.

Thus $\{\xi_1 \cdot \chi_{E_\lambda} \}$ is an uncountable family of pairwise orthogonal $L^2$ vectors whose finite sums have norm bounded by $\|\xi_1\|$, so all but countably many are zero.  The same is true for $\{\xi_2 \cdot \chi_{E_\lambda} \}$.  So for all but countably many $\lambda \in \C$, $E_\lambda$ is disjoint from $\supp(\xi_1) \cup \supp(\xi_2)=E$, of which it is a subset, making $E_\lambda$ empty as desired.
\end{proof}

\begin{corollary}\label{corCyclicProjections}
For any sequence of vectors $(\xi_n)_{n = 1}^\infty \subseteq L^2(X,\mu)$, there are   constants $(\lambda_n)_{n = 1}^\infty \subseteq \C$ such that for every $N \in \N$, we have \[\supp\Big(\sum_{n = 1}^N \lambda_n \xi_n\Big) = \Big( \bigcup_{n = 1}^N \supp(\xi_n) \Big). \]
\end{corollary}

\begin{proof}
Apply Lemma \ref{lemCyclicProjections} repeatedly.
\end{proof}

\section{Operators with controlled support expansion,  the discrete case}\label{SectionSuppExpDiscrete}
In this section we consider the counting measure on $\N$, so the measure of a set $E \subseteq \N$ is its cardinality $|E|$.  We discuss support expansion \emph{sequences}, the operators in $\mathcal{B}(\ell^2(\N))$ which are  controlled by them, and the \cstar-algebras which they generate. As such, this section deals with the discrete case of the sections to follow, and it contains the  preliminary work and motivation for them. 

We start by recalling some basic definitions and terminology. Given a Hilbert space $H$, the bounded operators on $H$ are denoted by $\mathcal{B}(H)$ and its ideal of compact operators by $\mathcal K(H)$. The Hilbert space of square summable $\C$-valued sequences is denoted by $\ell^2(\N)$, and we denote its standard unit basis by $(\delta_n)_{n\in\N}$. The support of $\xi\in \ell^ 2(\N)$ is just the set of indices where $\xi$ is nonzero.  Given an operator $a\in \mathcal{B}(\ell^2(\N))$, we   represent $a$ as an $\N$-by-$\N$ matrix by letting $a=[a_{n,m}]_{n,m\in\N}$, where $a_{n,m}=\langle a\delta_m,\delta_n\rangle$ for all $n,m\in\N$. We typically omit the outside indices and write $[a_{n,m}]$ for $[a_{n,m}]_{n,m}$.

\subsection{Uniformly RC-finite operators} We now recall the definition of the  \cstar-algebra studied  in \cite{Manuilov2019}  (see also Remark \ref{RemarkURA}).

\begin{definition}\label{defRCFin}
We say that an operator $a  \in \mathcal{B}(\ell^2(\N))$ is \textit{uniformly row and column finite} (abbreviated as \textit{uniformly RC-finite})  if there exists some $N \in \N$ such that the standard representation of $a$ as an $\N$-by-$\N$ matrix  $a=[a_{n,m}]$ of complex numbers  has at most $N$ nonzero entries per row and per column, i.e., 
\[ |\{n\in \N : a_{n,m}\neq 0\}|,|\{n\in \N : a_{m,n}\neq 0\}|\big\}\leq N, \qquad \forall m \in \N. \]
We denote the set of all uniformly RC-finite operators by  $B_{\mathrm{RC}}$.
\end{definition}

 It is not hard to check that if $[a_{n,m}]$ is a complex $\N$-by-$\N$ matrix which is uniformly row and column finite in the sense above, then $[a_{n,m}]$ corresponds to a bounded operator on $\ell^2(\N)$ if and only if   $\sup_{n,m\in\N}|a_{n,m}|<\infty$ (see \cite[Lemma 4.27]{RoeBook}).

 We next notice that the operators in  $B_{\mathrm{RC}}$ are precisely the ones which expand the support of vectors in $\ell^2(\N)$ by at most by a fixed linear factor. 
 
\begin{lemma}\label{lemAltRCFinDef}
An operator $a \in \mathcal{B}(\ell^2(\N))$ is uniformly RC-finite if and only if     there exists   $N \in \N$ such that \[|\supp(a\xi)|   \leq N |\supp(\xi)| \text{ and } |\supp(a^{*}\xi)| \leq N |\supp(\xi)|, \qquad \forall \xi \in \ell^2(\N). \]
\end{lemma}

\begin{proof} 
If the displayed inequalities hold, then  $|\supp(a\delta_n)|, |\supp(a^{*}\delta_n)| \leq N  $ for all $n\in\N$. This means that both $a$ and $a^*$ have  at most $N$ nonzero entries in each  of their columns, so $a$ is RC-finite.

On the other hand, suppose $a=[a_{n,m}]\in \mathcal{B}(\ell^2(\N))$ is uniformly RC-finite, with at most $N$ nonzero entries in each row and column.  This means $|\supp(a\delta_n)| \leq N$ and $|\supp(a^*\delta_n)| \leq N$ for all $n\in\N$. By Lemma \ref{musupportprops}(2) we get the desired inequalities for all finitely supported $\xi$.  If $|\supp(\xi)|= \infty$ either $N >0$ and the inequalities are trivial because the right-hand side is $\infty$, or $N=0$ and the inequalities are trivial because $a=0$.
 \end{proof}

It is straightforward from the definition that $B_{\mathrm{RC}}$  is closed under adjoints and scalar multiplication.  It is also easy to check that it is closed under addition (use Lemma \ref{musupportprops}(2)) and composition (Lemma \ref{lemAltRCFinDef}).  Therefore  $ B_\mathrm{RC}$ is a  $*$-subalgebra of $\mathcal{B}(\ell^2(\N))$, and its norm closure, which we denote by  $C_{\mathrm{RC}}$,  is a $C^{*}$-algebra.  

\begin{remark}\label{RemarkURA}
Alternatively, $C_{\mathrm{RC}}$ can be described as the \emph{uniform Roe algebra associated to the maximal uniformly locally finite coarse structure of $\N$}. This coarse structure is often denoted by $\cE_{\max}$, and hence $C_{\mathrm{RC}}$ is often denoted by $\cstu(\N,\cE_{\max})$ in the literature. We justify our choice for the somewhat unusual notation  $C_{\mathrm{RC}}$ by the fact that  this paper does not make use of the theory of  coarse spaces and of uniform Roe algebras, and we  refer the interested reader to \cite{RoeBook} for a detailed monograph   on coarse spaces and uniform Roe algebras. 
\end{remark}

It is evident from the definition that $B_{\mathrm{RC}}$ contains all operators $a=[a_{n,m}]\in \mathcal{B}(\ell^2(\N))$ such that $a_{n,m}\neq 0$ for only finitely many $(n,m)\in \N^2$. By a standard approximation argument, $C_{\mathrm{RC}}$ contains all compact operators.  It also contains noncompact operators, for instance the identity or any other permutation operator on $\ell^2(\N)$.  But $C_{\mathrm{RC}}$ is strictly smaller than $\mathcal{B}(\ell^2(\N))$. Indeed, let $a=[a_{n,m}]$ be an operator whose $m$th column has $m$ nonzero coordinates all equal to $1/\sqrt{m}$, and such that the supports of all columns are pairwise disjoint.  Then $a \in \cB(\ell^2(\N)) \setminus C_{\mathrm{RC}}$ (see   \cite[Proposition 2.1]{Manuilov2019} for details).

Summarizing the discussion above:

\begin{proposition}\label{thmManAlgNonTriv}
$\mathcal{K}(\ell^2(\N)) \subsetneq C_{\mathrm{RC}} \subsetneq \cB(\ell^2(\N))$.
\end{proposition}

\subsection{Support expansion sequences and their generated $C^{*}$-algebras}

By Lemma $\ref{lemAltRCFinDef}$, the uniformly RC-finite operators are precisely the ones which   expand the cardinality of the support of vectors by at most a fixed linear factor. This motivates the study of   operators   which expand  supports in other controlled manners, as well as the operator algebras generated by them. We point out that the framework of this section will be followed nearly \textit{verbatim} in the following sections for operators in $\mathcal{B}(L^2(\R))$. Therefore, the reader should view this section as establishing vocabulary and providing contrast for the material to follow.

\begin{definition}\label{defSuppExpSeq} Given an operator $a \in \mathcal{B}(\ell^2(\N))$, we define $\Phi_a\colon  \overline{\N} \to  \overline{\N}$, the \textit{support expansion sequence of $a$}, as \[\Phi_a(n) = \sup \Big\{|\supp(a\xi)| : \xi\in \ell^2(\N)\ \text{ and }\ |\supp(\xi)| \leq n \Big\}, \qquad \forall n \in \overline{\N}.\]
\end{definition}

From its definition the support expansion sequence of $a \in  \mathcal{B}(\ell^2(\N))$ is increasing.  The next lemma gathers a few other properties that will be useful.

\begin{lemma}\label{lemSupportExpansionSequenceProperties}
Let $\Phi_a$ be the support expansion sequence of $a \in \mathcal{B}(\ell^2(\N))$.
\begin{enumerate}
\item\label{ItemlemSupportExpansionSequenceProperties1} $\Phi_a(n)\leq \Phi_a(1)n$ for all $n\in\overline{\N}$.
\item \label{itemlemSupportExpansionSequenceProperties2} If   $\Phi_a(n_1) = \Phi_a(n_2)$ for some $n_1<n_2$ in $  \overline{\N}$, then    $\Phi_a(n_1) = \Phi_a(n_3)$  for all   $n_3 > n_1$ in $  \overline{\N}$.
\item\label{ItemlemSupportExpansionSequenceProperties4}  $\Phi_a(\infty) = \lim_{n \to \infty} \Phi_a(n)$.
\item\label{itemlemSupportExpansionSequenceProperties3} The following are equivalent:
\begin{itemize}
    \item $\Phi_a$ is bounded on $\N$;
    \item $\Phi_a(\infty)< \infty$;
    \item $\Phi_a$ repeats some finite value;
    \item $\Phi_a(n) < n$ for some $n \in \N$;
    \item $\Phi_a$ is eventually a finite constant.
\end{itemize}
These conditions imply that $a$ is finite rank. 
\end{enumerate}
\end{lemma}
\begin{proof}
\eqref{ItemlemSupportExpansionSequenceProperties1}: This follows from the proof of Lemma \ref{lemAltRCFinDef}.

\eqref{itemlemSupportExpansionSequenceProperties2}: Say $n_1 < n_2 < n_3\in  \overline{\N}$ and  $\Phi_a(n_1) = \Phi_a(n_2)$. Pick  $\xi\in \ell^2(\N)$ such that $|\supp(\xi)|\leq n_1$ and  $\Phi_a(n_1)=|\supp(a\xi)|$. As $\Phi_a(n_1)=\Phi_a(n_2)$,   $a$ must have no non-zero entries in any row outside  the support of $a\xi$, so   $\Phi_a(n_3) = \Phi_a(n_1)$.

\eqref{ItemlemSupportExpansionSequenceProperties4}: If $\Phi_a$ is bounded on $\N$, then it must repeat values, so by \eqref{itemlemSupportExpansionSequenceProperties2} it is eventually constant and keeps this constant value at $\infty$.  If $\Phi_a$ is unbounded on $\N$, then both sides of the desired equation are $\infty$. 

\eqref{itemlemSupportExpansionSequenceProperties3}: These statements follow from the previous items and their proofs. 
\end{proof}

\begin{example}\label{exSuppExpSeqConcUp}
Support expansion sequences do not need to be  concave down. Indeed, let  \[a = \begin{bmatrix}
    1       & 0 & 0 & 1 \\
    1       & 0 & 0 & 0 \\
    0       & 1 & 0 & 1 \\
    0       & 1 & 0 & 0  \\
    0       & 0 & 1 & 1 \\
    0       & 0 & 1 & 0  
\end{bmatrix}\] and, considering the canonical inclusion of the $6$-by-$4$ matrices in $\mathcal{B}(\ell^2(\N))$, view   $a$ as an element in $\mathcal{B}(\ell^2(\N))$. Then,   $\Phi_a = (0, 3, 4, 6, 6, \dots)$, which is not concave down. 
\end{example}

The next lemma gathers some ways that support expansion sequences interact with algebraic operations.  The second item follows from Lemma \ref{musupportprops}(2).  For the third,
\[|\supp(ab \xi)| \leq \Phi_a(|\supp(b \xi)|) \leq \Phi_a (\Phi_b(|\supp(\xi)|)),\]
where the latter inequality uses that $\Phi_a$ is increasing.

\begin{lemma}\label{lemSuppExpSeqCompAlgOps} Let $a, b \in \mathcal{B}(\ell^2(\N))$ and $\lambda \in \C\setminus\{0\}$. 
\begin{enumerate}
    \item\label{ItemlemSuppExpSeqCompAlgOps1}  $\Phi_{\lambda a} = \Phi_a$.
    \item\label{ItemlemSuppExpSeqCompAlgOps2} $\Phi_{a + b} \leq \Phi_a + \Phi_b$.
    \item\label{ItemlemSuppExpSeqCompAlgOps3} $\Phi_{ab} \leq \Phi_a \circ \Phi_b$.  
\end{enumerate}
\end{lemma}

\begin{definition}\label{defSeqContrOpers}
Given an increasing function $s\colon\overline{\N} \to \overline{\N}$, we define\[ B_s = \Big\{ a \in \mathcal{B}(\ell^2(\N)) : \Phi_a, \Phi_{a^{*}} \leq s \Big\}. \]
We call $B_s$   the set of  operators \emph{controlled by $s$}. 
\end{definition}

The next corollary follows immediately from Lemma   \ref{lemSuppExpSeqCompAlgOps}.

\begin{corollary}\label{corSeqContrOpersCompAlgOps}
Let $s_1, s_2$ be increasing functions $\overline{\N} \to \overline{\N}$, $a \in B_{s_1}$, $b \in B_{s_2}$, and $\lambda \in \C$. 
\begin{enumerate}
    \item $a^{*} \in B_{s_1}$.
    \item $\lambda a \in B_{s_1}$.
    \item $a + b \in B_{s_1 + s_2}$.
    \item $ab \in B_{s_1 \circ s_2}$.
\end{enumerate}
\end{corollary}
 
\begin{definition}\label{defFamSeqContrOpers}
Let  $\mathcal{S}$ be  a nonempty family of increasing maps  $\overline{\N}\to \overline{\N}$ which is   closed under addition and composition. We define the \textit{$\mathcal{S}$-controlled operators} by $B_{\mathcal{S}} = \cup_{s \in \mathcal{S}} B_s$.  By Corollary \ref{corSeqContrOpersCompAlgOps}, $B_\mathcal{S}$ is a *-algebra.  We denote the norm closure of $B_\mathcal{S}$ by $C_\mathcal{S}$ and call it the \emph{support expansion \cstar-algebra  generated by $\mathcal{S}$}.\footnote{Notice that  it is important to require $\mathcal S$ to be nonempty, as otherwise $C_{\mathcal S}$ would be the empty set. \label{FootnoteNonempty}}
\end{definition}

Given an   arbitrary nonempty set $\mathcal S$ of increasing maps $\overline{\N}\to \overline{\N}$, we denote by $\langle \mathcal S \rangle$ the   smallest family of maps $\overline{\N}\to \overline{\N}$    containing $\mathcal S$ which is   closed under addition and composition. If $\mathcal S$ consists of a singleton, say $s$, we simply write $\langle s\rangle$ for $\langle \{s\}\rangle$. By Corollary \ref{corSeqContrOpersCompAlgOps}, $B_{\langle \mathcal S \rangle}$   is a $*$-subalgebra of $\mathcal{B}(\ell^2(\N))$ and thus $C_{\langle \mathcal S \rangle}$ is a $C^{*}$-algebra. We call $C_\mathcal{\langle S\rangle }$  the \emph{$C^{*}$-algebra  generated by $\mathcal{S}$}.

\begin{example}\label{exDiscContrExpC*Algs} Consider an increasing function $s\colon \overline{\N}\to \overline{\N}$.
\begin{enumerate}
    \item If $s$ is the zero map,   then $\brak{\{s\}}=\{s\}$ and  $C_{\brak{s}}=\{0\}$.
    \item If $s$ is a nonzero finite constant function, then  $\brak{s}=\{(n+1)s : n\in \N\}$ and   $B_{\brak{s}}$ is the set of operators $a=[a_{n,m}]$ which have at most finitely many nonzero entries. We have $C_{\brak{s}} = \mathcal{K}(\ell^2(\N))$.
    \item If  $s(n) = n$ for every $n \in \overline{\N}$, then $\brak{s}$ is the set of linear sequences through the origin with positive integer slope, so $C_{\brak{s}}=C_{\mathrm{RC}}$.
    \item If $s(n) =\infty$ for every $n \in \overline{\N}$ then clearly $C_{\brak{s}}=B_{\brak{s}} = \mathcal{B}(\ell^2(\N))$.
\end{enumerate}
\end{example}
 
In fact, Example \ref{exDiscContrExpC*Algs} exhausts  all possible support expansion \cstar-subalgebras of $\mathcal{B}(\ell^2(\N))$. 

\begin{theorem}[Theorem \ref{thmDiscContrExpC*AlgsINTRO} in Section \ref{SectionIntro}]
For any nonempty family $\mathcal{S}$ of increasing maps $\overline{\N}\to \overline{\N}$, \[C_\mathcal{\brak{S}} \in \Big\{ \{0\}, \mathcal{K}(\ell^2(\N)), C_{\mathrm{RC}}, \mathcal{B}(\ell^2(\N)) \Big\}.\]\label{thmDiscContrExpC*Algs}
\end{theorem}

\begin{proof}
We have that $\brak{\mathcal{S}}$ is a nonempty family of increasing maps $\overline{\N}\to \overline{\N}$ closed under addition and composition. The following facts are straightforward: 
\begin{itemize}
\item If all $s \in \brak{\mathcal{S}}$ have $s(1)=0$, then $C_{\brak{S}} = \{0\}$.
\item If $s(1)=\infty$ for some $s\in \brak{S}$, then $C_{\langle S\rangle }= \mathcal{B}(\ell^2(\N))$.
\end{itemize}
We are left to analyze the case where all $s \in \brak{\mathcal{S}}$ have $s(1)<\infty$, and there is $s' \in \brak{\mathcal{S}}$ with $s'(1)>0$.  We will show that $  C_{\langle \mathcal{S}\rangle}$ must equal either  $\mathcal K(\ell^2(\N))$ or $ C_{\mathrm{RC}}$, depending on whether or not there is $a\in B_{\brak{\mathcal{S}}}$ with  $\Phi_a$ unbounded.

Case 1: There are $s \in \brak{\mathcal{S}}$ and $a \in B_s$ such that $\Phi_a$ is unbounded.  By Lemma \ref{lemSupportExpansionSequenceProperties}\eqref{itemlemSupportExpansionSequenceProperties3} we have $\Phi_a(n)\geq n$ for all $n\in\N$. Hence $s(n)\geq n$ for all $n\in\N$, and by Lemma   \ref{lemAltRCFinDef} we have that $  B_{\mathrm{RC}}\subseteq B_{\langle s\rangle}$. This shows that  $C_{\mathrm{RC}}\subseteq C_{\langle \mathcal{S}\rangle}$.   

Now take any $s\in \brak{\mathcal{S}}$ and $a\in B_s$. By Lemma \ref{lemSupportExpansionSequenceProperties}\eqref{ItemlemSupportExpansionSequenceProperties1}, 
\[\Phi_a(n), \Phi_{a^*}(n) \leq \max\{\Phi_a(1), \Phi_{a^*}(1)\} \cdot n \leq s(1) n, \forall n\in \overline{\N}.\]  As   $s(1)<\infty$, Lemma  \ref{lemAltRCFinDef} implies that $a\in B_{\mathrm{RC}} \subset C_{\mathrm{RC}}$. Taking norm limits of such $a$ gives $C_{\langle \mathcal{S}\rangle} \subseteq C_{\mathrm{RC}}$.  Combined with the previous paragraph, we have equality.

Case 2: $\Phi_a$ is bounded for all $a\in B_{\brak{\mathcal{S}}}$.  Then all elements of $B_{\brak{\mathcal{S}}}$ have finite rank (Lemma \ref{lemSupportExpansionSequenceProperties}\eqref{itemlemSupportExpansionSequenceProperties3}), and it follows that $C_{\langle \mathcal{S}\rangle }\subseteq \mathcal K(\ell^2(\N))$.

Recalling that $s' \in \brak{\mathcal{S}}$ has $s'(1) > 0$, any matrix unit $e_{ij}$ belongs to $B_{s'}$.  Thus $C_{\langle \mathcal{S}\rangle }$ is a \cstar-algebra containing all matrix units, so it contains $\mathcal{K}(\ell^2)$.   Combined with the previous paragraph, we have equality. 
\end{proof}

\section{Support expansion functions, the continuous case}\label{SectionSuppExpCont}
This section concerns the continuous version of support expansion sequences introduced in Definition \ref{defSuppExpSeq}.  Now we work with $\cB(L^2(\R)) = \cB(L^2(\R,\mu))$, where $\mu$ is the Lebesgue measure on $\R$.  The results obtained here will be used later  to analyze support expansion \cstar-subalgebras of $\cB(L^2(\R))$. 

\subsection{Basic  properties of support expansion functions} Here is the continuous version of   Definition \ref{defSuppExpSeq}. 

\begin{definition}\label{defSuppExpFunc}
Given an operator $a \in \mathcal{B}(L^2(\R))$, we define $\Phi_a\colon [0,\infty]\to [0,\infty]$, the \textit{support expansion function of $a$}, as \[\Phi_a(x) = \sup \Big\{\mu(\supp(a\xi)) : \xi \in L^2(\R)\ \text{ and }\  \mu(\supp(\xi)) \leq x \Big\}, \qquad \forall x \in [0,\infty].\]
\end{definition}

Just as with ``$\supp$", ``$\Phi$" is only meaningful once the underlying Hilbert space has been represented as an $L^2$ space, and a more complete notation would explicitly include the $L^2$ space.  In this paper we are only considering $\ell^2(\N)$ and $L^2(\R)$, and context will make it clear whether $\Phi$ has its meaning as a support expansion sequence or function.


From the definition, support expansion functions are increasing.  The following should be compared with  Lemma \ref{lemSuppExpSeqCompAlgOps}.

\begin{lemma}\label{lemSuppExpFuncCompAlgOps} Let $a, b  \in \mathcal{B}(L^2(\R))$, $ (a_i)_i$ be a sequence in $\mathcal{B}(L^2(\R))$, and  $\lambda \in \C\setminus \{0\}$.
\begin{enumerate}
    \item\label{ItemlemSuppExpFuncCompAlgOps1} $\Phi_{\lambda a} = \Phi_a$.
    \item \label{ItemlemSuppExpFuncCompAlgOps2} $\Phi_{a + b} \leq \Phi_a + \Phi_b$.
    \item\label{ItemlemSuppExpFuncCompAlgOps3} $\Phi_{ab} \leq \Phi_a \circ \Phi_b$.
    \item\label{ItemlemSuppExpFuncCompAlgOps4} If $a_i$ converges to $a$ in the strong operator topology (meaning $a_i \xi \to a\xi$ for all $\xi \in L^2(\R)$), then $\Phi_a \leq \liminf_i \Phi_{a_i}$.
    \item\label{ItemlemSuppExpFuncCompAlgOps5} $ \Phi_a(\infty)=\lim_{t \to \infty} \Phi_a(t) $.
\end{enumerate}
\end{lemma}
\begin{proof}  
Items \eqref{ItemlemSuppExpFuncCompAlgOps1},  \eqref{ItemlemSuppExpFuncCompAlgOps2}, and \eqref{ItemlemSuppExpFuncCompAlgOps3} are established by the same straightforward calculations as in Lemma \ref{lemSuppExpSeqCompAlgOps}.

For item \eqref{ItemlemSuppExpFuncCompAlgOps4},  fix $x\in [0,\infty]$ and choose any $\xi\in L^2(\R)$ with $\mu(\supp(\xi))\leq x$. As  $a\xi=\lim_i a_i\xi$, we use Lemma \ref{musupportprops}(1) to compute
\[\mu(\supp(a\xi))\leq \liminf_i\mu(\supp(a_i\xi)) \leq \liminf_i \Phi_{a_i}(x).\] Since $\xi$ is arbitrary, we are done.

For item \eqref{ItemlemSuppExpFuncCompAlgOps5}, let $y$ be any real number less than $\Phi_a(\infty)$.  Then there is some $\xi \in L^2(\R)$ with $\mu(\supp(a\xi)) > y$.  Let $\xi_t$ be the truncation $\xi \cdot \chi_{[-t/2,t/2]}$, so that $\xi_t \to \xi$ in $L^2(\R)$ and $\mu(\supp(\xi_t)) \leq t$.  Using Lemma \ref{musupportprops}(1) and the fact that $\Phi_a$ is increasing,
\[\Phi_a(\infty) \geq \lim_{t \to \infty} \Phi_a(t) = \liminf_{t \to \infty} \Phi_a(t) \geq \liminf_{t \to \infty} \mu(\supp(a \xi_t)) \geq \mu(\supp(a \xi)) > y. \]
Since $y$ was any real number less than $\Phi_a(\infty)$, it follows that $\lim_{t \to \infty} \Phi_a(t) = \Phi_a(\infty)$.
\end{proof}
We note later (Corollary \ref{corSuppExpContin}) that $\Phi_a$ is also continuous on $(0,\infty)$ (but not necessarily at 0; e.g., $\Phi_a$ is discontinuous at $0$ for any rank one operator $a$).

\smallskip

Definition \ref{defSuppExpFunc} of support expansion functions is vector focused. We now present a more algebraic,  projection-focused version which turns out to coincide with it (Theorem \ref{thmSuppExpFuncEquivDef}).

We will freely identify any function $f \in L^\infty(\R)$ with the corresponding multiplication operator on $L^2(\R)$, so that $L^\infty(\R)$ is a subalgebra (in fact a maximal abelian von Neumann subalgebra) of $\cB(L^2(\R))$.  Projections in $L^\infty(\R)$ are (equivalence classes of) characteristic functions, and we denote the set of them by $\mathcal{P}(L^\infty(\R))$.  The \emph{support projection} of $\xi\in L^2(X,\mu)$ is 
\[s(\xi)=\bigwedge \{p\in \mathcal{P}(L^\infty(\R)) :  p\xi=\xi \} \in L^\infty(\R) \subseteq \cB(L^2(\R)).\]
It is easy to check that $s(\xi)$ is nothing but the characteristic function of $\supp(\xi)$.  Defining a trace $\tau$ on $L^\infty(\R)$ by $\tau(f) = \int f\: d\mu$, we have $\tau(s(\xi)) = \mu(\supp(\xi))$.  In the rest of the paper, we often meet projections that are not naturally described in terms of sets, so we favor $\tau$ over $\mu$ in our computations.

Now for any $a\in \mathcal{B}(L^2(\R))$, we define its \emph{left support} by
\[s_l(a) = \bigwedge \{p \in \mathcal{P}(L^\infty(\R)) : pa = a \} \in L^\infty(\R) \subseteq \cB(L^2(\R)).\]
Again we note that the notions of support projection and left support rely on the $L^2$ function representation of the underlying Hilbert space, and actually they make sense when $L^\infty(\mathbb{R})$ is replaced with any unital von Neumann subalgebra of $\cB(H)$.  Since we only use them here in the context of $L^\infty(\R) \subseteq \cB(L^2(\R))$, we do not burden the notations with more information.

 
Finally we define  $\Phi'_a\colon [0,\infty]\to [0,\infty]$ by letting
\[ \Phi^{'}_a(x) = \sup \Big\{\tau(s_l(ap)) : p \in \mathcal{P}(L^\infty(\R))\ \text{ and }\ \tau(p) \leq x \Big\}. \]

For later use we collect some elementary left support identities that are mostly parallel to well-known facts about range projections.

\begin{lemma}\label{lemleftsupportfacts}
Let $a \in \cB(L^2(\R))$.
\begin{enumerate}
\item $s_l(aa^*) = s_l(a)$.
\item If $a$ is self-adjoint, then $a s_l(a) = a$.
\item If $S \subseteq H$ has dense span,
\[s_l(a) = \vee_{\xi \in H} s(a\xi) = \vee_{\xi \in S} s(a\xi).\]
\item $\Phi'_a(\infty) = \tau(s_l(a))$.
\end{enumerate}
\end{lemma}

\begin{proof}
(1): Since $s_l(a) a a^* = a a^*$, $s_l(a) \geq s_l(aa^*)$.  Also \[(s_l(aa^*)a - a)(s_l(aa^*)a - a)^* = s_l(aa^*)aa^*s_l(aa^*) - s_l(aa^*) aa^* - aa^*s_l(aa^*) + aa^* = 0, \] which implies $s_l(aa^*)a - a = 0$ and so $s_l(aa^*) \geq s_l(a)$.

(2): We have $as_l(a) = (s_l(a)a)^{*} = a^{*} = a$.

(3): The left support of $a$ is the smallest projection in $L^\infty(\R)$ fixing all vectors in the range of $a$, which is the same as fixing a subset with dense span.

(4): From (3), the map $\mathcal{P}(L^\infty(\R)) \ni p \mapsto s_l(ap)$ is order-preserving, so $\Phi'_a(\infty) = \tau(s_l(a I))$.
\end{proof}

\begin{theorem}\label{thmSuppExpFuncEquivDef}
For $a \in \mathcal{B}(L^2(\R))$, we have $\Phi_a = \Phi^{'}_a$.
\end{theorem}

\begin{proof}
Fix $a\in \mathcal{B}(L^2(\R))$ and $x\in (0,\infty]$;  we will show that $\Phi_a(x) = \Phi_{a}^{'}(x)$.  (This is obvious if $x=0$, as both sides are zero.)

$\Phi'_a(x) \geq \Phi_{a}(x)$: Let $y$ be a real number less than $\Phi_a(x)$.  Find a vector $\xi$ with $\tau(s(\xi)) \leq x$ and $\tau(s(a\xi)) > y$.  From Lemma \ref{lemleftsupportfacts}(3) applied to $a s(\xi)$,
\[s_l(a s(\xi)) \geq s(a s(\xi) \xi) = s(a \xi) \quad \Rightarrow \quad \tau(s_l(a s(\xi))) \geq \tau(s(a \xi)) > y \quad \Rightarrow \quad \Phi'_a(x) > y.  \]
The conclusion follows because $y$ is arbitrary.
  
$\Phi_a(x) \geq \Phi'_a(x)$: Let $p$ be a nonzero projection in $L^\infty(\R)$ with $\tau(p) \leq x$, and let $(\xi_n)_{n=1}^\infty$ be an orthonormal basis of $pL^2(\R)$.  By Corollary \ref{corCyclicProjections}, choose complex constants $(\lambda_j)$ so that $s(\sum_{n=1}^N \lambda_n a \xi_n) = \vee_{n=1}^N s(a \xi_n)$ for all $N$.  For any $N$ we have
\begin{align*}
\Phi_a(x) &\geq \Phi_a(\tau(p)) \geq \Phi_a\left( \tau\left(s\left(\sum_{n=1}^N \lambda_n \xi_n\right)\right)\right) \geq \tau\left(s\left(a \sum_{n=1}^N \lambda_n \xi_n\right)\right) \\ &= \tau\left(s\left(\sum_{n=1}^N \lambda_n a \xi_n\right)\right) = \tau\left(\bigvee_{n=1}^N s(a \xi_n)\right).\end{align*}
The finite joins in the last term increase with $N$ to the infinite join, which is $s_l(ap)$.  The trace $\tau$ (which corresponds to the Lebesgue integral) commutes with increasing limits.  This gives $\Phi_a(x) \geq \tau(s_l(ap))$,
and as $p$ was arbitrary we get the conclusion.
\end{proof}

\subsection{A tale of three families}
 
The goal of this subsection is to identify a  property that is necessary, and another that is sufficient, for a function $[0,\infty]\to [0,\infty]$ to be of the form $\Phi_a$ for some $a\in \mathcal{B}(L^2(\R))$. 
This culminates in   Theorem \ref{ThmTaleOfThreeFamilies}.

For any function $f\colon [0, \infty] \to  [0, \infty]$, we let $\overline{f} \colon [0, \infty] \to  [0, \infty]$ denote the map \[\overline{f}(x) = \left\{\begin{array}{ll}
\frac{f(x)}{x},& \text{ if } x\neq 0, \infty\\
\infty,& \text{ if } x=0\\
0, & \text{ if } x=\infty.
\end{array} \right. \] 
\begin{definition}\label{defICOD}
Consider a function $f\colon [0, \infty] \to  [0, \infty]$ with $f(0)=0$ and $\lim_{t \to \infty} f(t) = f(\infty)$. 
\begin{enumerate}
\item We say that  $f$ is $\ICOD$ (for  ``increasing and  concave down'') if it is   increasing and  concave down.
\item We say that  $f$ is $\ISOD$ (for  ``increasing and slope-to-origin decreasing'') if it is   increasing and $\overline f$ is   decreasing.
\item  We say that  $f$ is $\SUPPEXP$ (for  ``support expansion'') if $f=\Phi_a$ for some $a\in \mathcal{B}(L^2(\R))$.
\end{enumerate}
By abuse of notation, we also denote the subsets of all functions $ [0, \infty] \to  [0, \infty]$ which are $\ICOD$, $\ISOD$, and $\SUPPEXP$ by $\ICOD$, $\ISOD$, and $\SUPPEXP$, respectively.
\end{definition}

The following proposition gathers some simple properties of  $\ICOD$ and $\ISOD$, all proved by short computations.  For (2), an example of a function that is $\ISOD$ but not $\ICOD$ is $\max\{\sqrt{x}, x\}$.

\begin{proposition}\label{propISODClosedUnderAddAndComp}
${}$
\begin{enumerate}
\item\label{propISODClosedUnderAddAndComp1} Both sets  $\ICOD$ and  $\ISOD$   are  closed under addition and composition. 
\item\label{propISODClosedUnderAddAndComp2}  We have  $\ICOD \subsetneq \ISOD$.
\item\label{propISODClosedUnderAddAndComp3}  If $f\in \ISOD$, then either $f(x)=\infty$ for all $x>0$ or $f(x)<\infty$ for all $x<\infty$.
\end{enumerate}
\end{proposition}

The remainder  of this subsection is dedicated to the proof of the following result. 

\begin{theorem}\label{ThmTaleOfThreeFamilies} The following inclusions hold:  $\ICOD \subseteq \SUPPEXP \subseteq\ISOD$.
\end{theorem}

Although Theorem  \ref{ThmTaleOfThreeFamilies} does not give us a complete characterization of support expansion functions --- we may present results related to this elsewhere --- it will be useful in our study of support expansion \cstar-subalgebras of $\mathcal{B}(L^2(\R))$, undertaken in Sections \ref{SubsectionTop}, \ref{SubSectionUltimateElementsP}, and \ref{SectionLargePoset}.

We start by showing the first inclusion of Theorem \ref{ThmTaleOfThreeFamilies}. The construction in this proposition, witnessing prescribed $\ICOD$ support expansion by a weighted composition operator, will also be important in Section \ref{SectionAlgebrasContExp}.

\begin{proposition}\label{propICODAreSuppExp}
Let $f \colon [0,\infty]\to [0,\infty]$ and   $r\in (0,\infty]$. Assume that $f$   is $\ICOD$, strictly increasing on $[0,r]$, (right) continuous at 0, and constant on $[r,\infty]$. Then the formula 
\[(a_{f,r}\xi)(x) = \left\{\begin{array}{ll}
\sqrt{{(f^{-1})}^{'}(x)} \xi(f^{-1}(x)), &\text{ for a.e. } x\in [0,f(r)]\\
0,& otherwise
\end{array}\right. \]
defines a bounded operator on $L^2(\R)$, and  $\Phi_{a_{f,r}} = f$.  
\end{proposition}

\begin{proof}
Since $f\colon [0,r]\to [0,f(r)]$ is strictly increasing,  $f^{-1}\colon [0,f(r)]\to [0,r] $ is a well defined  strictly increasing concave up function, in particular absolutely continuous and a.e. differentiable. Applying the change of variables $y=f^{-1}(x)$, we  obtain $\|a_{f,r}\xi\|_2=\|\xi \cdot \chi_{[0,r]}\|_2 \leq \|\xi\|_2$ for all $\xi\in L^2(\R)$.\footnote{The reader can find the details in many standard measure theory books, for instance   \cite[Chapter 7]{RudinComplexBook} for change of variables and    \cite[Theorem A]{RobertsVarbergBook} for the fact that concave functions are absolutely  continuous.} So $a_{f,r}$ is   a well-defined bounded operator. 
  
It is ``visually obvious" from $f$ being $\ICOD$ that $a_{f,r}$ achieves maximal support expansion on vectors supported on intervals of the form $[0,x]$; here is the mathematical justification.  Let $x \in [0,\infty]$, and take   $\xi\in L^2(\R)$ supported on a set of size $\leq x$.  Then $a_{f,r} \xi$ is supported on $f(\supp \, \xi \cap [0,r])$, and \[\mu(\supp \, a_{f,r}(\xi)) = \mu \circ f(\supp \, \xi \cap [0,r]) = \int_{\supp \, \xi \cap [0,r]} f' \: d\mu.\]  As $f'$ is decreasing by concavity, the maximum value that can be obtained on the right-hand side, conditioned on $\mu(\supp \, \xi) \leq x$, occurs when $\supp \, \xi \cap [0,r]$ is $[0,x] \cap [0,r]$.  So $\Phi_{a_{f,r}}(x)$ is $f(x)$ when $x \leq r$ and $f(r)$ when $x \geq r$, i.e., $\Phi_{a_{f,r}} = f$.
\end{proof}

\begin{corollary}\label{CorpropICODAreSuppExp}
The inclusion $\ICOD \subseteq \SUPPEXP$ holds.
\end{corollary}

\begin{proof}
Let $f\in \ICOD$. As $f$ is concave down,   either $f(x)=\infty$ for all $x>0$ or $f(x)<\infty$ for all $x<\infty$. If the former holds, let $a\in \cB(L^2(\R))$ be the rank one projection onto a vector $\xi\in L^2(\R)$ with infinite support, so that $\Phi_a=f$. 

Suppose now that $f(x)<\infty$ for all $x<\infty$. As $f\in \ICOD$, there is $r\in [0,\infty]$ such that $f$ is strictly increasing on $[0,r]$ and constant on $[r,\infty]$. If $\lim_{t\to 0}f(t)=0$, the result follows from Proposition \ref{propICODAreSuppExp}. If not, let
\[g=f-\lim_{t\to 0}f(t).\]
Then $\Phi_{g,r}=g$ by Proposition \ref{propICODAreSuppExp}. Let $a\in\cB(L^2(\R))$ be the operator given by 
\[a\xi=a_{g,r}\xi+\langle \xi,\chi_{[-1,0]}\rangle\chi_{[-\lim_{t\to 0}f(t),0]}\]
for all $\xi\in L^2(\R)$. It is immediate to check that $\Phi_a=f$ by evaluating on vectors $\chi[0,x] + \eta$, where $\eta$ is a vector with arbitrarily small support inside $[-1,0]$.
\end{proof}

In order to finish the proof of Theorem \ref{ThmTaleOfThreeFamilies}, we are left to show the inclusion $\SUPPEXP\subseteq \ISOD$. For that,  we need some preliminary results. The following elementary facts about commuting projections will be used in the next proposition.

\begin{fact}\label{facCommProjTechLems} Given $p, q, q_1, q_2 \in \mathcal{P}(L^\infty(\R))$ with $q_1 \leq q_2$, we have 
\begin{enumerate}
    \item $\tau(p \vee q_1) - \tau(q_1) \geq \tau(p \vee q_2) - \tau(q_2)$, and
    \item $s_l(a(p \vee q)) = s_l(ap)\vee s_l(aq)$ for any $a \in \mathcal{B}(L^2(\R))$.
\end{enumerate}
\end{fact} 
  
\begin{proposition}\label{propSuppExpSlopDec1}
Let $a \in \mathcal{B}(L^2(\R))$, $x \in (0, \infty)$, and  $n \in \N$. Then, if   $y = \frac{n + 1}{n}x$, we have $\overline{\Phi_a}(x) \geq \overline{\Phi_a}(y)$.
\end{proposition}
\begin{proof}
Suppose for the sake of contradiction that $\overline{\Phi_a}(y) > \overline{\Phi_a}(x)$. Then  \[\Phi_a(y) = \overline{\Phi_a}(y) \cdot y > \frac{1}{2}(\overline{\Phi_a}(y) + \overline{\Phi_a}(x)) \cdot y.\] By the definition of $\Phi'_a = \Phi_a$ (Theorem \ref{thmSuppExpFuncEquivDef}), we can  can  pick  a projection $p \in \mathcal{P}(L^\infty(\R))$ such that $\tau(p) \leq y$ and \[\tau(s_l(ap)) > \frac{1}{2}(\overline{\Phi_a}(y) + \overline{\Phi_a}(x)) \cdot y.\] 
As $\tau(p)\leq y=\frac{n+1}{n}x$, we can write    $p=\sum_{k=1}^{n+1}p_k$ for some orthogonal sequence  $(p_k)_{k = 1}^{n + 1}$ in $\mathcal{P}(L^\infty(\R))$, all of which have trace at most $\frac{x}{n}$. 

For any $k \leq n$ we have that 
\begin{align*}\tau\Big( \bigvee_{i=1}^k s_l(ap_i) \Big) - \tau\Big(\bigvee_{i=1}^{k-1} s_l(ap_i)\Big) &= \tau\Big( s_l(ap_k) \vee \bigvee_{i=1}^{k-1} s_l(ap_i) \Big) - \tau\Big(\bigvee_{i=1}^{k-1} s_l(ap_i)\Big) \\ &\overset{\ref{facCommProjTechLems}.1}{\geq}
\tau\Big( s_l(ap_k) \vee \bigvee_{i \neq k} s_l(ap_i) \Big) - \tau\Big(\bigvee_{i \neq k} s_l(ap_i)\Big) \\
&\overset{\ref{facCommProjTechLems}.2}{=}
\tau( s_l(ap)) - \tau\Big(\bigvee_{i \neq k} s_l(ap_i)\Big) \\
&> \frac{1}{2}(\overline{\Phi_a}(y) + \overline{\Phi_a}(x)) \cdot y - \Phi_a(x).
\end{align*}
(When $k=1$, indices running from 1 to $k-1$ are the empty set.)

Add up these inequalities for $1 \leq k \leq n$, noting that the left-hand side telescopes:

\begin{align*}
    \tau\Big(\bigvee_{k = 1}^n s_l(ap_k)\Big) > n \left( \frac{1}{2}(\overline{\Phi_{a}}(y) + \overline{\Phi_{a}}(x)) \cdot y - \Phi_{a}(x) \right).
\end{align*}
The left-hand side is $\leq \Phi_a(x)$.  On the right-hand side, we substitute $y=\frac{n+1}{n}x$ and use our assumption that $\overline{\Phi_a}(y) > \overline{\Phi_a}(x)$:
\begin{align*}
\Phi_a(x) > n \left(  \overline{\Phi_{a}}(x) \cdot \frac{n+1}{n}x - \Phi_{a}(x) \right) 
= n   \left( \frac{\Phi_{a}(x)}{x}  \cdot \frac{n+1}{n}x - \Phi_{a}(x)  \right) 
 = \Phi_{a}(x).
\end{align*}
 So $\Phi_a(x)>\Phi_a(x)$, a contradiction.
\end{proof}

\begin{corollary}\label{corSuppExpSlopDec2}
Let  $a \in \mathcal{B}(L^2(\R))$, $x \in (0, \infty)$, and $q \in \Q$ with  $q \geq 1$. Then we have $\overline{\Phi_a}(x) \geq \overline{\Phi_a}(qx)$.
\end{corollary}

\begin{proof}
Say $q=\frac{r}{s}$ for naturals $r>s$. Then \[q=\frac{r}{r-1}\cdot\frac{r-1}{r-2}\cdot\ldots\cdot\frac{s+1}{s}\] and the result follows by applying  Proposition \ref{propSuppExpSlopDec1} repeatedly.
\end{proof}

The ``slope-to-origin'' non-decreasing property in Corollary \ref{corSuppExpSlopDec2} implies that the function $\Phi_a$ must be continuous on $(0,\infty)$.

\begin{proposition}\label{propSlopeDecContin} Let $f \colon [0, \infty] \to [0, \infty]$ be increasing and such that   $\overline{f}(x) \geq \overline{f}(qx)$ for all $x\in [0,\infty]$ and all $q \in \Q$ with $q \geq 1$. Then $f$ is continuous on $(0,\infty)$.
\end{proposition}

\begin{proof}
If $f(x)=\infty$ for all $x>0$, the result is trivial. So we can assume that $f(x)<\infty$ for all $x>0$ (see Proposition \ref{propISODClosedUnderAddAndComp}).  Let $q_j$ be rationals decreasing to 1; for any $x > 0$,
\[f(q_j x) = q_jx \bar{f}(q_jx) \leq q_jx \bar{f}(x) = q_j f(x).\]
Letting $j \to \infty$ and recalling that $f$ is increasing, we see that $\lim_{t \to x^+}f(t) = f(x)$.  The limit from the left can be proved similarly, starting with $f(\frac{x}{q_j})$.
\end{proof}

  Corollary \ref{corSuppExpSlopDec2} and Proposition \ref{propSlopeDecContin}, plus Lemma \ref{lemSuppExpFuncCompAlgOps}\eqref{ItemlemSuppExpFuncCompAlgOps5}, immediately give the following:

\begin{corollary}\label{corSuppExpContin} Let  $a \in \mathcal{B}(L^2(\R))$. Then $\Phi_a$ is continuous on $(0,\infty]$, and $\overline{\Phi_a}$ is   continuous on $ (0, \infty)$.
\end{corollary}

\begin{proof}[Proof of Theorem \ref{ThmTaleOfThreeFamilies}]
After Corollary \ref{CorpropICODAreSuppExp}, it is left  to show that $\SUPPEXP\subseteq \ISOD$. Fix $a\in \mathcal{B}(L^2(\R))$. We know that $\Phi_a$ is increasing, that $\Phi_a(0)=0$, and that $\lim_{t \to \infty} \Phi_a(t) = \Phi_a(\infty)$ (Lemma \ref{lemSuppExpFuncCompAlgOps}\eqref{ItemlemSuppExpFuncCompAlgOps5}).  We only need that $\overline{\Phi_a}$ is  decreasing, which follows from Corollaries \ref{corSuppExpSlopDec2} and  \ref{corSuppExpContin}.
\end{proof}

We conclude this subsection by applying the results above to obtain an interesting interaction between $\Phi_a$ and $\Phi_{a^*}$ that will be used later.  Notice that a support expansion function can have arbitrarily slow growth, since it can be any $\ICOD$ function (Corollary \ref{CorpropICODAreSuppExp}).  In contrast, the next proposition says if the graphs of $\Phi_a$ and $\Phi_{a^*}$ are \textit{both} below the line $y=x$ at some point, then they must be horizontal from there to the right.  In particular, for self-adjoint $a$ either $\Phi_a(x) \geq x$ for all $x \in [0,\infty]$, or $\Phi_a$ is constant from the point where its graph crosses below $y=x$.

Let us set some notation.  We know that for any $a \in \cB(L^2(\R))$, $\Phi_a$ is $\ISOD$, so the set $\{x \in (0,\infty): \, \Phi_a(x) < x\}$ has the form $(r_a, \infty)$ for some $r_a$.  We set $r_a  =\infty$ if this set is empty.  By continuity (Corollary \ref{corSuppExpContin}) $\Phi_a(r_a) = r_a$.  We also set $s_a = \max\{r_a, r_{a^*}\}$.
 
\begin{proposition}\label{propSelfAdjointSuppExpCharacterization}
Keep the notation from above for $a \in \cB(L^2(\R))$.
\begin{enumerate}
\item Let $a$ be self-adjoint.  We have $\Phi_a(x) = r_a$ for $x \in [r_a, \infty]$.
\item\label{propSelfAdjointSuppExpCharacterization.Item2} Remove the self-adjointness assumption.  We have $\Phi_a(x) = \Phi_a(s_a)$ for $x \in [s_a, \infty]$.
\end{enumerate}
 
\end{proposition}

\begin{proof}
(1): First assume that $\tau(s_l(a))<\infty$. As $\Phi_a$ is   increasing,   it is enough to show that  $\Phi_a(\infty) = \Phi_a(r_a)$.   Suppose not, so by Lemma \ref{lemleftsupportfacts}(4) $\tau(s_l(a))= \Phi_a(\infty) > \Phi_a(r_a) = r_a$.  From the definition of $r_a$,  $\Phi_a(\tau(s_l(a)))<\tau(s_l(a))$. 
We compute
\[\tau(s_l(a)) \overset{\text{Lemma \ref{lemleftsupportfacts}(2)}}{=} \tau(s_l(as_l(a)))  \leq \Phi_a(\tau(s_l(a))) <\tau(s_l(a)),\]
which is a contradiction.

Next we remove the condition $\tau(s_l(a))<\infty$.  Let $p_j = \chi_{[-j,j]} \in \mathcal{P}(L^\infty(\R))$.  We notice that $\Phi_{p_j a p_j} \leq \Phi_a$: this follows by computing, for any $q \in \mathcal{P}(L^\infty(\R))$, 
\[ \tau(s_l(p_j a p_j \cdot q)) \leq \tau(s_l(a \cdot p_j q)) \leq \Phi_a(\tau(p_j q)) \leq \Phi_a(\tau(q)).\]
From $\Phi_{p_j a p_j} \leq \Phi_a$ it follows that $r_{p_j a p_j} \leq r_a$.  Since $\tau(s_l(p_j a p_j)) \leq \tau(p_j) = 2j < \infty$, the previous paragraph applies to $p_j a p_j$, so  $\Phi_{p_j a p_j}$ must equal $r_{p_j a p_j}$ on $[r_{p_j a p_j},\infty]$.  In particular $\Phi_{p_j a p_j}(\infty) = r_{p_j a p_j} \leq r_a$.  Since $p_j a p_j \to a$ in the strong operator topology, we have by Lemma \ref{lemSuppExpFuncCompAlgOps}\eqref{ItemlemSuppExpFuncCompAlgOps4} that \[r_a = \Phi_a(r_a) \leq \Phi_a(\infty) \leq \liminf \Phi_{p_j a p_j}(\infty) = \liminf r_{p_j a p_j} \leq r_a.\]  The conclusion follows from $\Phi_a(\infty) = \Phi_a(r_a) = r_a$.

(2):  If $s_a =\infty$, the statement is trivial, so we may assume $s_a < \infty$.
Take any finite $x > s_a$, so that $\Phi_a(x), \Phi_{a^*}(x) < x$.  Then  \[\Phi_{aa^*}(x) \leq \Phi_a(\Phi_{a^*}(x)) \leq \Phi_a(x) < x.\] (The middle inequality used the fact that $\Phi_a$ is increasing.)  This implies that $r_{aa^*} \leq s_a$.  Since $aa^*$ is self-adjoint, by (1) $\Phi_{aa^*}$ is the constant $r_{aa^*}$ on $[r_{aa^*}, \infty]$, which includes $[s_a, \infty]$.  Letting $x$ decrease to $s_a$ in the chain of inequalities above, we have 
\[r_{aa^*} = \Phi_{aa^*}(x) \leq \Phi_a(x) \searrow \Phi_a(s_a),\] giving $r_{aa^*} \leq \Phi_a(s_a)$.  (Here if $s_a >0$ we used the continuity of $\Phi_a$, Corollary \ref{corSuppExpContin}.  If $s_a =0$ then $0 \leq \Phi_a(x) \leq x$ for all $x$, so $\lim_{x \to 0^+} \Phi_a(x) = 0 = \Phi(s_a)$.)  
Now compute \[\Phi_a(s_a) \leq \Phi_a(\infty)= \tau(s_l(a)) \overset{\text{Lemma \ref{lemleftsupportfacts}(1)}}{=} \tau(s_l(aa^*))= \Phi_{aa^*}(\infty) =r_{aa^*} \leq \Phi_a(s_a).\] The conclusion follows from $\Phi_a(\infty) = \Phi_a(s_a)$.
\end{proof}

\begin{example} \label{Propsuppexpfcnsconstant}
This example illustrates Proposition \ref{propSelfAdjointSuppExpCharacterization} and shows the difference between its two parts.

Define $f: [0,\infty] \to [0,\infty]$ by $f(x) = \min\{\frac{x}{2}, 1\}$.  Proposition \ref{propICODAreSuppExp} constructs $a = a_{f,2} \in \cB(L^2(\R))$ with $\Phi_a = f$; explicitly, $a \xi(x) = \sqrt{2} \cdot f(2x) \cdot \chi_{[0,1]}(x)$.  One can compute that $a^* \xi(x) = \frac{1}{\sqrt{2}} \cdot \xi(\frac{x}{2}) \cdot \chi_{[0,2]}(x)$, so $\Phi_{a^*}(x) = \min\{2x,2\}$.

See Figure \ref{Phicollapse}: $\Phi_a$ need not become constant until both $\Phi_a$ and $\Phi_{a^*}$ are below $y=x$.
\end{example}

\begin{figure}[h]

\includegraphics[height = 3in, width = 3in]{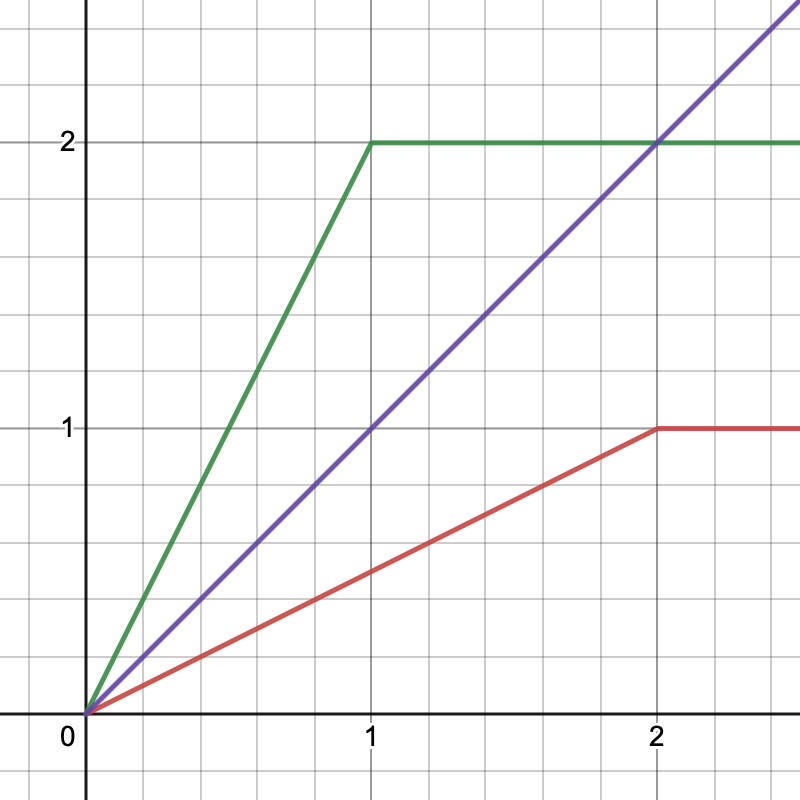} 
\caption{\label{Phicollapse}The functions considered in Example \ref{Propsuppexpfcnsconstant}. The graph of $\Phi_a = f$ (red) is below $y=x$ on $(0,\infty)$, so $r_a=0$.  But it does not become constant until both it and the graph of $\Phi_{a^*}$ (green) go below $y=x$ at $s_a =2$, as required by Proposition \ref{propSelfAdjointSuppExpCharacterization}\eqref{propSelfAdjointSuppExpCharacterization.Item2}.}
\end{figure}
\subsection{More on the relation between   $\ICOD$ and $\ISOD$}

  Theorem \ref{ThmTaleOfThreeFamilies} motivates a deeper study of $\ICOD$ and $\ISOD$ in order to better understand $\SUPPEXP$. Although $\ICOD\subseteq \ISOD$ is a strict inclusion, this subsection establishes  the following two results: 
  
  \begin{itemize}\item every member of $\ISOD$ is the supremum of members of $\ICOD$ (Proposition \ref{propISODEnvelopeICOD});
  \item   given $f\in \ISOD$, there is a procedure for obtaining $g\in \ICOD$ with $f\leq g\leq 2f$ (see Definition \ref{defConcaveConjugate} and Proposition \ref{propISODDoubleConjugateBounds}).
  \end{itemize}

It is well known that concave down functions, and thus $\ICOD$ functions, are continuous except perhaps at endpoints. The same holds for functions in   $\ISOD$, by Proposition \ref{propSlopeDecContin}.  The next proposition isolates some trivial facts about $\ICOD$ and $\ISOD$ functions:

\begin{proposition}\label{propISODClosedSupremums}
Let $(f_i)_{i \in I}$ be a nonempty family of functions $[0,\infty]\to [0,\infty]$. 
\begin{enumerate}
\item If $(f_i)_{i \in I}\subseteq \ICOD$, then the  infimum of $(f_i)_{i \in I}$ belongs to $\ICOD$.
\item If $(f_i)_{i \in I}\subseteq \ISOD$, then  both the  infimum and supremum of  $(f_i)_{i \in I}$ belong to $\ISOD$.
\end{enumerate}
\end{proposition}

Although $\ICOD$ is not  closed  under   supremum (e.g., $x\mapsto\max\{\sqrt{x},x/2\}$ is not concave down), the inclusion $\ICOD \subseteq \ISOD$ gives us that such a   supremum   belongs to $\ISOD$.   In fact, this characterizes the $\ISOD$ functions:

\begin{proposition}\label{propISODEnvelopeICOD}
Every $\ISOD$ function is the    supremum of a sequence of  $\ICOD$ functions.
\end{proposition}

\begin{proof}
Take $f \in \ISOD$.  The conclusion is clear if $f \equiv 0$, so assume not, which entails that $f(x) > 0$ for all $x > 0$. For each $q\in \Q_{+}$, let $f_q\colon [0,\infty]\to [0,\infty]$ be the $\ICOD$ function given by \[f_q(x) = \frac{f(q)}{q}x \cdot \chi_{[0, q]}(x) + f(q)\cdot \chi_{(q, \infty]}(x), \quad \forall x\in [0,\infty].\]   Note  $f_q(q)=f(q)$ and, as $\overline f$ is   decreasing,  $f \geq f_q$. Let  $g$ be the  supremum of  $(f_q)_{q\in \Q_+}$, so $g$ is $\ISOD$ by Proposition \ref{propISODClosedSupremums}. Therefore  $f$ and $g$ are both $\ISOD$ and agree on all rational points, so by continuity of $\ISOD$ functions (Proposition \ref{propSlopeDecContin}) we have  $f = g$.
\end{proof}

We now present a procedure which, given $f\in \ISOD$, finds $g\in \ICOD$ such that $f\leq g\leq 2f$.  For that,   we  employ a half-line version of the \textit{concave conjugate} (a variation of the well-studied \textit{convex conjugate} or \textit{Fenchel-Legendre Transform}):

\begin{definition}\label{defConcaveConjugate}
Given   $f \colon (0, \infty) \to  [-\infty, \infty]$, we define the  \textit{concave conjugate} of $f$ as the map $f_*\colon (0,\infty)\to [-\infty,\infty]$ given by
 \[f_{*}(x) = \inf_{\lambda \in (0, \infty)} (\lambda x - f(\lambda)), \quad \forall x \in (0, \infty).\]
\end{definition}

As the  infimum of lines with nonnegative slope, $f_{*}$ is increasing and concave down.

\begin{proposition}\label{propISODDoubleConjugateBounds}
If $f\colon (0, \infty) \to  (0, \infty)$ is an increasing function such that $f(x)/x$ is decreasing, then $f \leq f_{**} \leq 2f$.
\end{proposition}

\begin{proof}
First notice that, unfolding definitions, we have   \[f_{**}(x) =  \inf_{\lambda \in (0, \infty)}\Big( \lambda x + \sup_{y \in (0, \infty)} (f(y) - \lambda y)\Big)\] for all $x\in (0,\infty)$. Letting  $y = x$ above, we have $f_{**}(x) \geq  f(x)$. On the other hand, letting  $\lambda = f(x)/x$ above,  we have 
\begin{align*}
f_{**}(x)& \leq f(x) + \sup_{y \in (0, \infty)} \left(f(y) - y\frac{f(x)}{x}\right) = f(x) + \sup_{y \in (0, \infty)} y\left(\frac{f(y)}{y} - \frac{f(x)}{x}\right).
\end{align*} 
Since $f(x)/x$ is   decreasing, the supremum above only needs to consider $y\in [0,x]$. Therefore, using that $f$ is increasing, we have that 
\begin{align*}
f_{**}(x) &  \leq f(x) + \sup_{0<y\leq x} \left(f(y) - y\frac{f(x)}{x}\right)  \leq f(x) + \sup_{0 < y \leq x} f(y) = 2f(x).
\end{align*}
This finishes the proof.
\end{proof}

\begin{definition}\label{DefipropISODDoubleConjugateBounds}
Let $f\in \ISOD$. By abuse of notation, we denote by  $f_{**}$ the function $[0,\infty]\to [0,\infty]$ defined as  
\[f_{**}(x)=\left\{\begin{array}{ll}0,& x=0\\
   (f\restriction (0,\infty))_{**}(x),  & x\in (0,\infty)  \\
  \lim_{x\to \infty}  (f\restriction [0,\infty))_{**}(x),  & x=\infty.
\end{array}\right.\]
\end{definition}

Notice that if $f\in \ISOD$ is such that $f(x)=\infty$ for all $x>0$, then $f_{**}(x)=\infty$ for all $x>0$. Hence, by Proposition \ref{propISODDoubleConjugateBounds}, we have that $f\leq f_{**}\leq 2f$ on all $[0,\infty]$, and $f_{**}\in \ICOD$.

\subsection{Families generated by functions in $\ISOD$}\label{SubsectionFamiliesGen}
 Given a nonempty family $\cF$ of functions $[0, \infty]\to [0,\infty]$, we denote the smallest  family of maps   containing $\cF$ which is closed under addition and composition by  $\brak{\cF}$. If $\cF$ consists of a singleton, say $\cF=\{f\}$, we simply write $\langle f\rangle $ for $\langle \{f\}\rangle$. This section gathers some   useful properties about such families which will be helpful throughout the rest of the paper. 

\begin{proposition}\label{propISODSubAdditive}
Functions in  $\ISOD$ are subadditive.
\end{proposition}

\begin{proof}
Fix $f\in \ISOD$. For  $x, y \in (0, \infty)$ we have  that 
\begin{align*}
f(x) + f(y) &= x\overline{f}(x) + y\overline{f}(y) \geq x\overline{f}(x + y) + y\overline{f}(x + y) = (x + y)\overline{f}(x+y) = f(x+y).
\end{align*}  Subadditivity when either summand is $0$ or $\infty$ is straightforward.
\end{proof}

\begin{proposition}\label{propSinglyGenerateFuncFamDominatedByLinearCombs}
Let $f \in\ISOD$ and $f_0 \in \brak{f}$. Then there exists $K \in \N$ and $(n_k)_{k = 1}^K \subseteq \N$ such that $f_0 \leq \sum_{k = 1}^K n_k f^{(k)}$.
\end{proposition}

\begin{proof}
Note that $f_0$ is constructed from composing and adding $f$ a finite number of times. Recall that every element of $\brak{f}$ is in $\ISOD$  (Proposition \ref{propISODClosedUnderAddAndComp}) and thus subadditive (Proposition  \ref{propISODSubAdditive}). We then apply subadditivity repeatedly to obtain the result.
\end{proof}
  
In fact, the following more general result holds, and its proof is completely analogous to the proof of Proposition \ref{propSinglyGenerateFuncFamDominatedByLinearCombs}:

\begin{proposition}\label{propMultiGenerateFuncFamDominatedByLinearCombs}
Let $\cF\subseteq\ISOD$ and $f_0 \in \brak{\cF}$. Then $f_0$ is dominated by    some linear combination  with natural number coefficients of compositions of members of $\cF$.\qed
\end{proposition}
 
The following is a straightforward consequence of Proposition \ref{propMultiGenerateFuncFamDominatedByLinearCombs}. 
 
\begin{corollary}\label{corFuncDominationImpliesFamilyDomination}
If $\cF,\mathcal G  \subseteq\ISOD$ are such that for all $f\in \cF$ there exists $g\in \langle \mathcal G \rangle $ such that $f  \leq g$, then for any $f_0 \in \brak{\cF}$ there exists some $g_0 \in \brak{\mathcal G }$ such that $f_0 \leq g_0$.\qed
\end{corollary}

\section{The poset $\mathbb{P}$ of continuous support expansion \cstar-algebras}\label{SectionAlgebrasContExp}
 
 We   now properly define support expansion \cstar-subalgebras of $\mathcal{B}(L^2(\R))$, as done in Section \ref{SectionSuppExpDiscrete} for the discrete case of $\mathcal{B}(\ell^2(\N))$. We also provide technical results which will allow us to study the poset given by all these \cstar-algebras in Sections \ref{SubSectionUltimateElementsP} and \ref{SectionLargePoset}.  In Section \ref{SubsectionTop}, we study some \cstar-algebras which can be obtained simply by this method.

\subsection{Sets of controlled operators}\label{SubsectionContOp}
The next definition is the continuous version of Definition \ref{defSeqContrOpers}.

\begin{definition}\label{defFuncContrOpers}
Given an increasing map $f\colon [0, \infty] \to  [0, \infty]$, we define \[ B_f = \Big\{ a \in \mathcal{B}(L^2(\R)) : \Phi_a, \Phi_{a^{*}} \leq f \Big\}. \]
Operators in $B_f$ are said to be \emph{controlled by $f$}.
\end{definition}

Here are some elementary relations between sets of controlled operators (cf.\ Lemma \ref{corSeqContrOpersCompAlgOps}).

\begin{lemma}\label{lemFuncContrOpersCompAlgOps}
Let  $f_1, f_2  \colon [0, \infty] \to  [0, \infty]$ be increasing, $a \in B_{f_1}$, $b \in B_{f_2}$, and $\lambda \in \C$. Then:
\begin{enumerate}
    \item $a^{*} \in B_{f_1}$.
    \item $\lambda a \in B_{f_1}$.
    \item $a + b \in B_{f_1 + f_2}$.
    \item $ab \in B_{f_1 \circ f_2}$.
\end{enumerate}
\end{lemma}

Different functions may generate the same set of controlled operators. Therefore, it is useful to have methods which, given a map $f\colon [0,\infty]\to [0,\infty]$, produce $g\colon [0,\infty]\to [0,\infty]$ with ``better'' properties and such  that $B_f=B_g$. The next two results provide such methods. We start by showing that one can always assume that $f\in \ISOD$. For that, we need a definition:

\begin{definition}\label{defISODEnvelope}
For any $f\colon [0, \infty] \to  [0, \infty]$ with $f(0)=0$, we define the \textit{$\ISOD$ lower-envelope} of $f$ by \[\tilde{f}(x) = \sup \Big\{g(x): g \in \ISOD \ \text{ and } \ g \leq f \Big\}\]
   for all $x\in [0,\infty]$.  
\end{definition}

Clearly $\tilde{f} \leq f$ and,   by Proposition \ref{propISODClosedSupremums} it follows that   $\tilde{f}\in \ISOD$.

\begin{proposition}\label{propContFuncsMayAsWellBeISOD}
For any increasing function $f\colon [0, \infty] \to  [0, \infty]$, we have $B_{f} = B_{\tilde{f}}$.
\end{proposition}
\begin{proof}
We have $B_{\tilde{f}} \subseteq B_{f}$ because $\tilde{f} \leq f$. Now suppose  $a \in B_{f}$.  Since $\Phi_a$ and $\Phi_{a^{*}}$ are $\ISOD$ (Theorem \ref{ThmTaleOfThreeFamilies}) and $\leq f$, by definition they are $\leq \tilde f$. This gives us that  $a \in B_{\tilde{f}}$.
\end{proof}

When interested in $B_f$, one can replace a slowly-growing $f$ with a bounded one. 

\begin{lemma}\label{lemControlFunctionBelowHalfSlopeMayAsWellBeBounded}
Let  $f \in\ISOD$   and  $r = \inf\{ x : f(x) < x \}$. Then $B_f = B_{\min\{f, r\}}$.
\end{lemma} 

\begin{proof}
If $a \in \cB(L^2(\R))$ belongs to $B_f$, then $\Phi_a, \Phi_{a^*} \leq f$.  In the terminology of Proposition \ref{propSelfAdjointSuppExpCharacterization}, we have $s_a \leq r$, so by Proposition \ref{propSelfAdjointSuppExpCharacterization}.2 we conclude that $\Phi_a, \Phi_{a^*} = s_a \leq r$ on $[s_a, \infty]$ and thus $\Phi_a, \Phi_{a^*} \leq r$ everywhere.  It follows that $\Phi_a, \Phi_{a^*} \leq \min\{f,r\}$, so $a \in B_{\min\{f,r\}}$.  The other inclusion is obvious.
\end{proof}

\subsection{Algebras of controlled operators}

We now introduce the continuous version of  Definition \ref{defFamSeqContrOpers}:

\begin{definition}\label{defFamFuncContrOpers}
If $\mathcal{F}$ is a nonempty family of increasing functions   $[0, \infty]\to [0,\infty]$ closed under addition and composition, then we define the \textit{$\mathcal{F}$-controlled operators} by $B_{\mathcal{F}}  = \bigcup_{f \in \mathcal{F}} B_f$.  By Lemma \ref{lemFuncContrOpersCompAlgOps}, $B_\mathcal{F}$ is a *-algebra.  We denote the norm closure of $B_\mathcal{F}$ by $C_\cF$ and call it the \emph{support expansion \cstar-algebra  generated by $\cF$}.
\end{definition}

For the remainder of this paper, our goal is to understand the algebras $C_\cF$   above.  When is one such algebra equal to, or included in, another?  
We will develop methods to study the poset of all such algebras, so we give this object a name:

\begin{definition}\label{DefintionP}
We let $\mathbb P$ denote  the poset of all support expansion \cstar-subalgebras of $\mathcal{B}(L^2(\R))$, with order given by inclusion.  
\end{definition}

The following is the main result of this subsection.

\begin{theorem}
\label{ThmpropSinglyGeneratedC*AlgCountablyGeneratedByISOD}
Given a nonempty family $\cF$ of increasing maps $[0,\infty]\to [0,\infty]$, there is a family $\cG\subseteq \ICOD$ such that $C_{\langle \cG\rangle}=C_{\langle \cF\rangle}$.
\end{theorem}

 Before proving Theorem \ref{ThmpropSinglyGeneratedC*AlgCountablyGeneratedByISOD},
we need some auxiliary results.

\begin{proposition}\label{propSinglyGeneratedC*AlgCountablyGeneratedByISOD}
Given a nonempty family $\cF$ of increasing maps $[0,\infty]\to [0,\infty]$, there is a family $\mathcal G\subseteq \ISOD$ such that $C_{\langle \cF\rangle }=C_{\langle \mathcal G\rangle}$.  
\end{proposition}

\begin{proof}
Let $\cF$ be a set of increasing functions $[0,\infty]\to [0,\infty]$. Then 
Proposition  \ref{propContFuncsMayAsWellBeISOD} gives us that
\[
    B_{\brak{\cF}} =\bigcup_{g \in \brak{\cF}} B_g
    = \bigcup_{g \in \brak{\cF}} B_{\tilde{g}} \subseteq B_{\brak{\tilde{g} : g \in \brak{\cF}}}
     \subseteq B_{\brak{\cF}},
\]
where the final inclusion follows since any map in $\brak{\tilde{g} : g \in \brak{\cF}}$ is dominated by something in $\brak{\cF}$. Indeed, this is the case since   $\tilde{g_1} + \tilde{g_2} \leq g_1 + g_2$ and $\tilde{g_1}(\tilde{g_2}(x))   \leq g_1(g_2(x))$  for all functions $g_1,g_2\colon [0,\infty]\to [0,\infty]$.
 
So, we have that $B_{\brak{\cF}} = B_{\brak{\tilde{g} : g \in \brak{\cF}}}$ and thus $C_{\brak{\cF}} = C_{\brak{\tilde{g} : g \in \brak{\cF}}}$. As   $\tilde{g} \in\ISOD$ for each $g \in \brak{\cF}$, we are done.
\end{proof}

\begin{proposition}\label{propDominatedFunctionImpliesSmallerAlgebra}
Let $\cF, \cG \subseteq\ISOD$ be  such that   for all $f\in \cF $ there is   $g \in \brak{\cG}$ with $f \leq g$. Then, $B_{\brak{\cF}} \subseteq B_{\brak{\cG}}$.
\end{proposition}
\begin{proof}
Let $a \in B_{\brak{\cF}}$ and pick $f\in \brak{\cF}$ such that $\Phi_a, \Phi_{a^{*}} \leq f$.  By Corollary \ref{corFuncDominationImpliesFamilyDomination}, there is  $g  \in \brak{\cG}$ such that $f \leq g$. Then $a \in B_{g} \subseteq B_{\brak{\cG }}$. 
\end{proof}

It then follows that $C^{*}$-algebras  generated by finitely many $\ISOD$ functions are in fact  generated by a single $\ISOD$ function:
 
\begin{corollary}\label{corFinitelyGeneratedFromISODC*AlgIsSinglyGenerated}
Given $(f_n)_{n = 1}^N \subseteq\ISOD$, we have that  $B_{\brak{f_1, \dots, f_N}} = B_{\brak{\sum_{n = 1}^N f_n}}$.
\end{corollary}
\begin{proof}
Since $\sum_{n = 1}^N f_n \in \langle {f_1, \dots, f_N} \rangle$, it follows that $B_{\langle\sum_{n = 1}^N f_n\rangle} \subseteq B_{\langle f_1, \dots, f_N\rangle}$. On the other hand, $f_n \leq \sum_{n = 1}^N f_n$ for each $1 \leq n \leq N$, so Proposition \ref{propDominatedFunctionImpliesSmallerAlgebra} gives us that $B_{\langle{f_1, \dots, f_N\rangle}} \subseteq B_{\langle\sum_{n = 1}^N f_n\rangle}$, and this completes the proof.
\end{proof}
 
The next corollary shows that any $C^{*}$-algebra  generated by $\ISOD$ functions is also  generated by $\ICOD$ functions, and the generating sets can be taken to have the same cardinality. For simplicity of notation, we introduce the following: given a nonempty family $\cF$ of maps $[0,\infty]\to [0,\infty]$, we let 
\[\cF_{**}=\{f_{**}\colon f\in \cF\}.\]

\begin{corollary}\label{corContFuncsForC*AlgMayAsWellBeICOD}
If $\cF \subseteq\ISOD$, then $B_{\brak{\cF}} = B_{\brak{\cF_{**}}}$.
\end{corollary}
\begin{proof}
Proposition \ref{propISODDoubleConjugateBounds} and Definition \ref{DefipropISODDoubleConjugateBounds} give us that $f  \leq {f }_{**} \leq 2f $ for each $f\in \cF$. The  result then follows from Proposition \ref{propDominatedFunctionImpliesSmallerAlgebra}.
\end{proof}

\begin{proof}[Proof of Theorem \ref{ThmpropSinglyGeneratedC*AlgCountablyGeneratedByISOD}]
This follows from Proposition \ref{propSinglyGeneratedC*AlgCountablyGeneratedByISOD}
 and Corollary \ref{corContFuncsForC*AlgMayAsWellBeICOD}.
\end{proof}

 \subsection{The truncation and the interpolate of a function}
 
We now present two more methods of replacing a given nonempty family  of maps $\cF $ by a simpler family $\cG $ for which, under mild assumptions on $\cF $, we still have  $B_{\langle\cF \rangle} = B_{\langle\cG\rangle}$. 
 
We start with the truncation procedure: Given  $f \in\ISOD$, we let  $f_{\_}\colon [0,\infty]\to[0,\infty]$ be the function given by  \[f_{\_}(x) = \min(f(x), f(1))\]  for all $x\in [0,\infty]$. We  call $f_{\_}$ the  \textit{truncation} of $f$. Notice that $f_{\_} \in\ISOD$. Given a noempty family $\cF\subseteq \ISOD$, we let 
\[\cF_{\_}=\{f_{\_}\colon  f\in \cF\}.\]

\begin{proposition}\label{propDichotomyAtInfinityFlat}
Let $\cF \subseteq\ISOD$ be such that $\lim_{x \to  \infty} \frac{f (x)}{x} = 0$ for every $f\in \cF$. Then $B_{\brak{\cF}} = B_{\brak{\cF_{\_}}}$.
\end{proposition}

\begin{proof}
The inclusion $B_{\brak{\cF_{\_}}} \subseteq B_{\brak{\cF}}$ is trivial since   ${f }_{\_} \leq f $ for each $f\in \cF $ (Proposition \ref{propDominatedFunctionImpliesSmallerAlgebra}). For the other direction, take any $a \in B_{\brak{\cF}}$ and $f_0\in   \brak{\cF}$ such that  $\Phi_a, \Phi_{a^{*}} \leq f_0$. As $f_0$ is formed from a finite number of additions and compositions of elements in $\cF$, we have  $\lim_{x \to  \infty} \frac{f_0(x)}{x} = 0$.  In particular, $f_0(x)$ is eventually less than $x$. So Lemma \ref{lemControlFunctionBelowHalfSlopeMayAsWellBeBounded} gives us that $B_{f_0} = B_{\min(f_0, r)}$ where $r = \inf\{ x : f_0(x) < x \}<\infty$. In particular, $a\in B_{\min(f_0, r)} $, which implies that  $\Phi_a$ and $ \Phi_{a^{*}}$ are bounded by $r$.

Suppose there is $h\in \cF$ for which $\lim_{x \to  0}h(x) = t > 0$. Then for any $n \geq \frac{r}{t}$, for $x > 0$ we have $(nh)_{-}(x) \geq nt \geq r \geq \Phi_a(x), \Phi_{a^*}(x)$.  Therefore  $a \in B_{{(nh)}_{-}} \subseteq B_{\brak{\cF_{\_}}}$.
 
Suppose no such $h\in \cF$ exists. The function $f_0$ is built out of a finite number of additions and compositions of elements of $\cF$, so we can define   $g$ to  be the function built out of the corresponding additions and compositions of the corresponding elements of  $\cF_{\_}$; in particular  $g \in \brak{\cF_{\_}}$ is eventually constant. As $\lim_{x \to  0} f(x) = 0$ for every $f\in \cF$, there is $\delta>0$ so that  $f_0(x) = g(x)$ for all $x \in [0, \delta]$.\footnote{Notice that, since we are composing members of  $\cF_{\_}$, we cannot simply take $\delta=1$.} Therefore, $\Phi_a(x), \Phi_{a^{*}}(x) \leq g(x)$ for all   $x\in [0, \delta]$. On the other hand, as $\Phi_a$ and $ \Phi_{a^{*}}$ are bounded, there is   $n\in\N$ such that    $\Phi_a(x), \Phi_{a^{*}}(x) \leq n g(x)$ on $[\delta, \infty]$. Then $a \in B_{ng} \subseteq B_{\brak{\cF_{\_}}}$.
\end{proof}

Given $f \in\ISOD$, we  define the \emph{interpolate of $f$} as $f_{/}(x):[0,\infty]\to [0,\infty]$ as the function given by  
\[f_{/}(x)=\left\{
\begin{array}{ll}
f(x),& \text{ if }x\in [0,1],\\
f(1)x,& \text{ if } x\in (1,\infty].
\end{array}\right.
\]
Clearly,  $f_{/} \in\ISOD$. Given $\cF\subseteq \ISOD$, we let 
\[\cF_{/}=\{f_/\colon f\in \cF\}.\]

\begin{proposition}\label{propDichotomyAtInfinityLinear}
Let $\cF \subseteq\ISOD$ be  such that $\lim_{x \to  \infty} \frac{f (x)}{x} > 0$ for some $f\in \cF $. Then $B_{\langle\cF\rangle} = B_{\langle \cF_/\rangle}$.
\end{proposition}

\begin{proof}
The inclusion  $B_{\langle\cF\rangle} \subseteq B_{\langle\cF_/\rangle}$ is clear since  $f  \leq {f }_{/}$ for each $f\in \cF $ (Proposition \ref{propDominatedFunctionImpliesSmallerAlgebra}). For the other direction, fix $a \in B_{\langle\cF_/\rangle}$ and $f_0 \in \langle\cF_/\rangle$ with  $\Phi_a, \Phi_{a^{*}} \leq f_0$. As $f_0$ is formed from a finite number of additions and compositions of the elements in  $\cF_/$, there are   $\lambda,s>0$ so that $f_0(x) = \lambda x$ for $x \in [s, \infty]$. Let  $f\in \cF$  be as in the statement of the proposition. As $\lim_{x \to  \infty}  f (x)/x > 0$,  this discussion shows that, for all $\delta>0$ there is $n=n_\delta\in\N$ such that  $f_0(x) \leq nf (x)$ for $x \in [\delta, \infty]$.

Let $n=n_1$ and suppose there is $h\in \cF$ such that $\lim_{x \to  0}h(x) > 0$.   Then we can find $m \in \N$ such that $f_0(x) 
\leq mh(x)$ for $x \in [0, 1]$. Thus,   $f_0 \leq mh + nf $ and it follows that  $a \in B_{mh+ nf} \subseteq B_{\brak{\cF}}$.
 
Suppose now that no such $h\in\cF $ exists. The function    $f_0$ is built out of a finite number of additions and compositions of elements of $\cF_/$, so we can  define   $g$ to   be the function built out of the corresponding additions and compositions of the respective elements of  $\cF$. As $\lim_{x \to  0} f(x) = 0$ for every $f\in \cF $,  there is $\delta>0$ such that $f_0(x) =g(x)$ for all $x \in [0,\delta]$. Letting $n=n_\delta$, we have that    $f_0 \leq g + nf $ and thus $a \in B_{g + nf } \subseteq B_{\brak{\cF}}$.
\end{proof}

Intuitively, the content of the preceding two propositions is that as far as $B_\cF$ is concerned, the behavior of a family $\cF$ at infinity only comes in two classes.  If at least one member exhibits linear growth out to infinity, then we might as well assume that all its members do; otherwise we might as well assume that all its members are constant after 1.  We will see that the behavior of $\cF$ near 0 is encoded much more sensitively in $B_\cF$ (cf. Remark \ref{fullposet}).

 \subsection{The top and the bottom of $\mathbb P$}\label{SubsectionTop}
The simplest examples of elements in $\mathbb P$ are 
\[
C_{\langle 0\rangle}=\{0\}\ \text{ and }\  C_{\ICOD}=C_{\langle\infty\rangle}=\cB(L^2(\R)),\] where here $\infty$ denotes the function in $\ICOD$ which sends $0$ to $0$ and all other numbers to $\infty$. So $ C_{\langle 0\rangle}$ and $C_{\ICOD}$ are the first and last elements of $\mathbb P$, respectively. In this subsection, we show that the poset  $\mathbb P$ also has unique second and penultimate elements.

For the next result, let $x_{\_}$ denote the truncation of the map $f(x) =x$, i.e., $\min\{x,1\}$. 

\begin{proposition}\label{propSecondElement}
If $\cF\subseteq \ICOD$ contains a nonzero map, then $C_{\langle x_{\_}\rangle}\subseteq C_{\langle \cF\rangle}$. In particular,  $C_{\langle x_{\_}\rangle}$ is the unique immediate successor of $\{0\}$ in $\mathbb P$.
\end{proposition}

\begin{proof}
Fix a nonzero $f\in\cF$. Then, as $f(0)=0$ and $f$ is concave down, there is $n\in\N$ so that $x\leq nf(x)$ for all $x\in [0,1]$. So $C_{\langle x_{\_}\rangle}\subseteq C_{\langle \cF\rangle}$ (Proposition \ref{propDominatedFunctionImpliesSmallerAlgebra}). For the second statement, notice that, since it is enough to consider families in $\ICOD$ in order to analyse $\mathbb P$ (Theorem \ref{ThmpropSinglyGeneratedC*AlgCountablyGeneratedByISOD}), it follows that $C_{\langle x_{\_}\rangle}\subseteq C_{\langle \cG\rangle}$ for any arbitrary family of maps $\cG$ such that $\{0\}=C_{\{0\}}\subsetneq C_{\langle \cG\rangle}$. Therefore we only need to notice that $\{0\}\subsetneq C_{\langle x_{\_}\rangle}$, which follows since (multiplication by) $\chi_{[0,1]}$  clearly belongs to $C_{\langle {x\_}\rangle}$.
\end{proof}

For the next proposition, let \[\ICOD_{<\infty} = \Big\{ f \in\ICOD : f(x) < \infty\ \text{ for all }\ x < \infty \Big\}.\] Notice that the only member of $\ICOD \setminus \ICOD_{<\infty}$ is the function that is  $\infty$ for all $x>0$ (Proposition \ref{propISODClosedUnderAddAndComp}.3), and that $\langle \ICOD_{<\infty} \rangle = \ICOD_{<\infty}$.

\begin{proposition}\label{propTopPOSETFullNotSubsetOfFinite}
Let $g\colon [0,\infty]\to [0,\infty]$ be given by \[g(x)=\left\{\begin{array}{ll}
x+1,& \ x>0\\
0,&\ x=0. 
\end{array}\right.\]
If
$\cF\subseteq \ICOD_{<\infty}$, then   $C_{\langle \cF\rangle}\subseteq C_{\langle g \rangle}$. In particular,  $C_{\langle g \rangle}=C_{\ICOD_{<\infty}}$. Moreover, $C_{\ICOD_{<\infty}}$  is the unique immediate predecessor of $\cB(L^2(\R))$ in $\mathbb P$.
\end{proposition}

\begin{proof}
We first check that for $f \in \ICOD_{<\infty}$, we have $f \leq f(1) g$.  If $0 < x\leq 1$, then $f(x) \leq f(1) \leq f(1)(x+1)$.  If $x \geq 1$, then we use that $f$ is $\ISOD$ to compute $f(x) \leq f(1)x \leq f(1)(x+1)$.  Thus for any positive integer $n > f(1)$, we have $f \leq ng$ and so $B_f \subseteq C_{\langle g \rangle}$.  This suffices for the first two statements.


We now prove the last statement. Let $C_{\langle \cG\rangle}\subsetneq \cB(L^2(\R))$ be an arbitrary element in $\mathbb P$. By Theorem \ref{ThmpropSinglyGeneratedC*AlgCountablyGeneratedByISOD},
 we can assume that $\cG\subseteq\ICOD$.  As  $C_{\langle \cG\rangle}\neq \cB(L^2(\R))$, the function that is $\infty$ for all $x>0$ does not belong to $\langle G \rangle$, so  $\cG\subseteq\ICOD_{<\infty}$ and  $C_{\langle \cG \rangle}\subseteq C_{\ICOD_{< \infty}}$. 
 
 It is left to notice that $C_{\ICOD_{<\infty}}$ is not all of $ \cB(L^2(\R))$. For that, consider the operator $a \in \mathcal{B}(L^2(\R))$ given by \[a\xi = \sum_{n = 0}^\infty \frac{1}{\sqrt{2^n}} \brak{\xi, \chi_{[n, n + 1]}} \chi_{[2^n, 2^{n + 1}]} \ \text{ for all }\ \xi\in L^2(\R).\] Fix   $f \in B_{\ICOD_{< \infty}}$   and   $b \in B_f$. By the definition of $B_f$,  we have for any $n \in \N$ that \[ \tau(s(b\chi_{[n, n+1]}))\leq f(1).\] It follows that  
 \begin{align*}
 \norm{a - b}^2 &\geq \norm{a\chi_{[n , n + 1]} - b \chi_{[n, n + 1]}}^2 
 = \norm{\frac{1}{\sqrt{2^n}}\chi_{[2^n, 2^{n + 1}]} - b\chi_{[n, n + 1]}}^2
  \geq \frac{2^n - f(1)}{2^n}
  \end{align*}
  because subtracting $b \chi_{[n,n+1]}$ leaves the value of $\frac{1}{\sqrt{2^n}}\chi_{[2^n, 2^{n + 1}]}$ unchanged on a subset of size at least $2^n - f(1)$.  As $f(1)<\infty$ and $n\in\N$ is arbitrary, this shows that $d(a,B_f)\geq 1$. By the arbitrariness of $f$, this shows that  $d(a,B_{\ICOD_{<\infty}})\geq 1$. So   $a \notin C_{\ICOD_{<\infty}}$.  
\end{proof}
 
\section{A characterization of the order relation   in $\mathbb P$}\label{SectionCharFunct}
In this section we  continue developing machinery to study the poset $\mathbb P$ by  reducing the problem of determining when $C_{\langle \cF\rangle}$ is contained in $C_{\langle \cG\rangle}$   to a question about growth of functions at $0$ and $\infty$. Here is the main result of this section:

\begin{theorem}\label{thmCharacterizationOfMultiFGenC*AlgsFromFunctionProperties}
Let  $\cF,\cG \subseteq \ICOD$ be such that $\lim_{x \to  0} f (x) = 0$ for each $f\in \cF $. Then $C_{\brak{\cF}} \nsubseteq C_{\brak{\cG}}$ if and only if there exists $f_0\in \cF $ such that either
\begin{enumerate}
    \item $\lim_{x \to  \infty} \frac{f_0  (x)}{x} > 0$ and $\lim_{x \to  \infty} \frac{g (x)}{x} = 0$ for every $g\in \cG $, or
    \item for all nonzero $g_0 \in \brak{\cG}$ there is   $(x_n)_{n = 1}^\infty$ decreasing to $0$ such that $\lim_{n \to  \infty} \frac{f _0(x_n)}{g_0(x_n)} = \infty$.
\end{enumerate}
\end{theorem}

We need some preparatory results for Theorem 
\ref{thmCharacterizationOfMultiFGenC*AlgsFromFunctionProperties}.

\begin{proposition}\label{propSeparatingOperatorHasBoundedDistanceFromCollection}
Let $r\in (0,\infty]$, $g\in \ISOD$,  and $f\in \ICOD $  be such that  $f$ is  strictly increasing on  $[0, r]$ and   $\lim_{x \to  0}f(x) = 0$. Then, for $a_{f,r}$   given by Proposition \ref{propICODAreSuppExp}, we have \[d(a_{f, r}, B_g)^2 \geq 1 - \frac{g(x_0)}{f(x_0)}\] for every $x_0 \in (0, r]$, where $d$ is the metric induced by the operator norm.
\end{proposition}

\begin{proof}
Fix $x_0 \in (0, r]$. 
The vector given by \[\xi(x) = \frac{\sqrt{f^{'}(x)}}{\sqrt{f(x_0)}}\chi_{[0, x_0]}(x)\ \text{ for all } \ x\in \R,\] has norm 1 and 
\[(a_{f, r}\xi)(x) = \frac{1}{\sqrt{f(x_0)}} \chi_{[0, f(x_0)]}(x).\] Fix $b \in B_{g}$. Since   $b\xi$ has support of size at most $g(x_0)$, it follows that  
\begin{align*}
    \norm{a_{f, r} - b}^2 \geq \norm{(a_{f, r} - b)\xi}^2 &= \int_0^\infty  \left(\frac{1}{\sqrt{f(x_0)}}\chi_{[0, f(x_0)]}(x) - (b\xi)(x)\right)^2 dx \\ &\geq \frac{1}{f(x_0)}(f(x_0) - g(x_0))  \\ &= 1 - \frac{g(x_0)}{f(x_0)}.
\end{align*} 
Since this is true for an arbitrary $b \in B_g$,   the result follows.
\end{proof}

\begin{proposition}\label{propSeparatingOperatorInCorrespondingCollection}
Let $f \in \ICOD$,  $r>0$, and $\lambda > 0$ be such that $\lim_{x \to  0} f(x) = 0$ and $f'>\lambda$ on $[0,r)$.  (Again, derivatives need only exist a.e.)  Then $a_{f, r} \in B_{\brak{f}}$, where $a_{f, r}$ is as in  Proposition \ref{propICODAreSuppExp}.
\end{proposition}

Note that $f'>\lambda$ on $[0,\infty)$ if and only if $\lim_{x \to \infty} \frac{f(x)}{x} >\lambda$.

\begin{proof}

Let $f$, $r$, and $\lambda$ be as in the statement.  By Proposition \ref{propICODAreSuppExp}, we have that 
\[\Phi_{a_{f, r}}(x) = f(x)\cdot \chi_{[0, r]}(x) + f(r) \cdot \chi_{[r, \infty]}(x), \quad \forall x \in [0,\infty]\]
In particular, $\Phi_{a_{f, r}}\leq f$.

The issue is to bound $\Phi_{a_{f, r}^{*}}$.  For that, note that $a_{f, r}^{*}$ is the operator given by 
\[(a_{f, r}^{*}\xi)(x) = \sqrt{f^{'}(x)} \xi(f(x))\chi_{[0, r]}(x) \ \text{ for all }\ \xi\in L^2(\R)\ \text{ and all }\ x\in \R.\]
As in the proof of Proposition \ref{propICODAreSuppExp}, we have, for any vector $\xi \in L^2(\R)$,
\[\mu(\supp \, a^*_{f,r}(\xi)) = \mu \circ f^{-1}(\supp \, \xi \cap [0,f(r)]) = \int_{\supp (\xi) \cap [0,f(r)]} (f^{-1})' \: d\mu.\]  By assumption $(f^{-1})'(x) = \frac{1}{f'(f^{-1}(x))} \leq \frac{1}{\lambda}$ for $x \in [0,f(r))$.  For any $x \in [0,\infty]$, we take the supremum of the integral expression over $\xi$ with $\mu(\supp(\xi)) \leq x$ to deduce $\Phi_{a^*_{f, r}}(x) \leq \max\{\frac{x}{\lambda}, \frac{f(r)}{\lambda}\}.$  Since $f$ is a nonzero $\ICOD$ function, there is $n\in\N$ such that   $  \Phi_{a_{f, r}^*}\leq nf$.  Combined with the previous paragraph, we have $a_{f,r} \in B_{\brak f}$.
\end{proof}

We  now  consider how function behavior near $0$ affects the resulting $C^{*}$-algebras:

\begin{proposition}\label{propC*AlgSeperationFromFunctionDominationAtZeroMultiGen}
Let $\cF, \cG \subseteq \ICOD$, and suppose there is $f_0\in \cF$ such that $\lim_{x \to  0} f_0 (x) = 0$ and for every nonzero $g\in \langle\cG\rangle$ there is a sequence $(x_n)_{n = 1}^\infty$ decreasing to $0$ such that $\lim_{n \to  \infty} \frac{f_0(x_n)}{g(x_n)} = \infty$. Then $C_{\langle\cF\rangle} \nsubseteq C_{\langle\cG\rangle}$.
\end{proposition}

\begin{proof}
From the conditions, $f$ must be nonzero and finite.  Being ICOD there must be $r>0$ such that $f'_0$ is bounded away from 0 on $[0,r]$, so $a_{f_0, r}\in C_{\brak{\cF}}$ by Proposition \ref{propSeparatingOperatorInCorrespondingCollection}.  On the other hand, by Proposition \ref{propSeparatingOperatorHasBoundedDistanceFromCollection},   $a_{f_0, r}$ has distance $1$ from $B_{g }$ for any $g\in \langle\cG\rangle$, so  $a_{f_0, r} \notin C_{\brak{\cG}}$.
\end{proof}
 
For singly-generated support expansion \cstar-algebras, Proposition \ref{propC*AlgSeperationFromFunctionDominationAtZeroMultiGen} holds under weaker conditions: 

\begin{corollary}\label{corC*AlgSeperationFromFunctionDominationAtZeroSingleGen}
Let nonzero $f, g \in \ICOD$ and suppose   $\lim_{x \to  0} f(x) = 0$. If for every $N \in \N$ there is  $(x_n)_{n = 1}^\infty$ decreasing to $0$ such that $\lim_{n \to  \infty} \frac{f(x_n)}{g^{(N)}(x_n)} = \infty$, then $C_{\brak{f}} \nsubseteq C_{\brak{g}}$.
\end{corollary}

\begin{proof}
We first notice that $(x_n)_n$ in the statement can be chosen to work for all $N\in\N$ simultaneously. If $\lim_{x \to  0} \frac{g(x)}{x} \leq 1$, then $g$   dominates each $g^{(N)}$ near $0$, and our  claim is clear since the sequence $(x_n)_n$ given by $N=1$ works for all $N\in\N$. On the other hand, if $\lim_{x \to  0} \frac{g(x)}{x} > 1$ then each $g^{(N + 1)}$ dominates $g^{(N)}$ near $0$.  An easy diagonalization allows us to pick  $(x_n)_{n = 1}^\infty$ decreasing to $0$ such that  $\lim_{n \to  \infty} \frac{f(x_n)}{g^{(N)}(x_n)} = \infty$ for every $N \in \N$.
 
Fix  $g_0 \in \brak{g}$. By   Proposition \ref{propSinglyGenerateFuncFamDominatedByLinearCombs},  $g_0$ is dominated by some linear combination of the $g^{(N)}$, and so $\lim_{n \to  \infty} \frac{f(x_n)}{g_0(x_n)} = \infty$. The result then follows from  Proposition  \ref{propC*AlgSeperationFromFunctionDominationAtZeroMultiGen}.
\end{proof}
 
\begin{proof}[Proof of Theorem \ref{thmCharacterizationOfMultiFGenC*AlgsFromFunctionProperties}]
($\Leftarrow$) Suppose there is $f_0\in \cF$ such that either (1) or (2) holds for $f_0$. If the second item holds, the result follows from   Proposition \ref{propC*AlgSeperationFromFunctionDominationAtZeroMultiGen}. If the first item holds, then $(f_0)' \geq \lim_{x \to 0} \frac{f(x)}{x} > 0$, so $a_{f_0,\infty} \in C_{\brak \cF}$ by Proposition \ref{propSeparatingOperatorInCorrespondingCollection}.  But taking large enough $x_0$ in Proposition \ref{propSeparatingOperatorHasBoundedDistanceFromCollection}, we conclude $d(a_{f_0,\infty}, C_{\brak{\cG}}) \geq 1$.  Thus $C_{\langle\cF\rangle} \nsubseteq C_{\langle\cG\rangle }$.


$(\Rightarrow)$ We   prove this by contrapositive, so suppose both items fail. Suppose first that  $\lim_{x \to  \infty} \frac{f (x)}{x} = 0$ for every $f\in \cF $. Then $C_{\brak{\cF}} = C_{\langle \cF_{\_}\rangle}$ by Proposition \ref{propDichotomyAtInfinityFlat}. So there is no loss of generality to assume that each    $f\in \cF $ is  bounded. Fix   $f_0\in \cF$. Then, as the second item does not hold, there is $g_0 \in \brak{\cG}$, $\delta > 0$, and $n \in \N$ such that $f_0(x) \leq n g_0(x)$ for all $x \in [0, \delta]$. Therefore, as $f_0$ is bounded,   by replacing $n$ by a larger natural if necessary, we can assume that   $f_0 \leq ng_0$. Since $f_0\in \cF $ was arbitrary, Proposition \ref{propDominatedFunctionImpliesSmallerAlgebra} implies that $C_{\brak{\cF}} \subseteq C_{\brak{\cG}}$.

Suppose now that there exists $f\in \cF $ such that $\lim_{x \to  \infty} \frac{f(x)}{x} > 0$. Then, as the first item fails, there is  $g\in \cG$ such that $\lim_{x \to  \infty} \frac{g (x)}{x} > 0$. Hence, by Proposition \ref{propDichotomyAtInfinityLinear}, we have that $C_{\langle\cG\rangle} = C_{\langle\cG_/\rangle}$. So there is no loss of generality to assume that each  element of  $\cG $ is eventually linear. Fix some $f_0\in \cF$. As the second item does not hold, we can find $g_0 \in \brak{\cG}$, $\delta > 0$, and $n \in \N$ such that $f_0(x) \leq n g_0(x)$ for all $x \in [0, \delta]$.  As $\frac{f_0(x)}{x}$ is   decreasing, $f_0$ is at most asymptotically linear, so, replacing $n$ by a larger $n$ if necessary, we can assume that $f_0 \leq n g_0$. Since   $f_0\in \cF$ was arbitrary, Proposition \ref{propDominatedFunctionImpliesSmallerAlgebra} implies that $C_{\brak{\cF}} \subseteq C_{\brak{\cG}}$.
\end{proof}

With the aid of Corollary \ref{corC*AlgSeperationFromFunctionDominationAtZeroSingleGen}, Theorem \ref{thmCharacterizationOfMultiFGenC*AlgsFromFunctionProperties} takes the following form when $\cF$ and $\cG$ are singletons.

\begin{corollary}\label{corCharacterizationOfSingleFGenC*AlgsFromFunctionProperties.ICOD}
Let  nonzero $f, g \in \ICOD$ be such that $\lim_{x \to  0} f(x) = 0$. Then  $C_{\brak{f}} \nsubseteq C_{\brak{g}}$ if and only if either
\begin{enumerate}
    \item $\lim_{x \to  \infty} \frac{f(x)}{x} > 0$ and  $\lim_{x \to  \infty} \frac{g(x)}{x} = 0$ or
    \item for all $N \in \N$ there is $(x_n)_{n = 1}^\infty$ decreasing to $0$ such that $\lim_{n \to  \infty} \frac{f(x_n)}{g^{(N)}(x_n)} = \infty$.
\end{enumerate}
\end{corollary}

In fact,  $\ICOD$ can be replaced by $\ISOD$ in the previous corollary.   

\begin{corollary}\label{corCharacterizationOfSingleFGenC*AlgsFromFunctionProperties}
Let  nonzero $f, g \in \ISOD$ be such that $\lim_{x \to  0} f(x) = 0$. Then  $C_{\brak{f}} \nsubseteq C_{\brak{g}}$ if and only if either
\begin{enumerate}
    \item $\lim_{x \to  \infty} \frac{f(x)}{x} > 0$ and  $\lim_{x \to  \infty} \frac{g(x)}{x} = 0$ or
    \item for all $N \in \N$ there is $(x_n)_{n = 1}^\infty$ decreasing to $0$ such that $\lim_{n \to  \infty} \frac{f(x_n)}{g^{(N)}(x_n)} = \infty$.
\end{enumerate}
\end{corollary}

\begin{proof}
This is a consequence of Proposition \ref{propISODDoubleConjugateBounds} and Definition \ref{DefipropISODDoubleConjugateBounds}. Indeed, $f \leq f_{**} \leq 2f$ implies that $C_{\langle f\rangle }=C_{\langle f_{**}\rangle }$ and $C_{\langle g\rangle }=C_{\langle g_{**}\rangle }$, and also that the first item of the corollary holds for $f$ and $g$ if and only if it does for $f_{**}$ and $g_{**}$. As for the second item, first use that $g_{**}$ is ISOD to write $g_{**}(2^N x)\leq 2^N g_{**}(x)$ for all $x\geq 0$ and all $N\in\N$. Combined with $g \leq g_{**} \leq 2g$, we get by induction that $g^{(N)} \leq g_{**}^{(N)}\leq 2^N g^{(N)}$ for all $N\in\N$, and the second item of the corollary also holds for $f$ and $g$ if and only if it does for $f_{**}$ and $g_{**}$.
\end{proof}

\section{The last elements of  $\mathbb{P}$}\label{SubSectionUltimateElementsP}
We now study some  elements of $\mathbb P$ generated by  natural subsets of   $\ICOD$ and show that they are precisely  at ``the end'' of $\mathbb P$ (see Theorem \ref{thmTopPOSETContainments}).  

\begin{definition}\label{defExtremePOSETElements} Let $\ICOD_{<\infty}$ be as in the previous subsection.  We define:
\begin{enumerate}
\item\label{Item2}  $\ICOD_{\mathrm{bdd}} = \{ f \in\ICOD : f \text{ is bounded} \}$
\item\label{Item3}  $\ICOD_{0} = \{ f \in\ICOD : \lim_{x \to 0} f(x) = 0 \}$
\item\label{Item4}  $\ICOD_{0 \cap \mathrm{bdd}} =\ICOD_{\mathrm{bdd}} \cap\ICOD_{0}$
\end{enumerate}
\end{definition}

Note that $\brak{\mathcal{F}} = \mathcal{F}$ for $\cF$ being any of the families above.

\begin{proposition}\label{propTopPOSETBoundedNotSubsetOfLimitZero}
$C_{\ICOD_0}$ and $ C_{\ICOD_{\mathrm{bdd}}}$ are incomparable.
\end{proposition}

\begin{proof}
We first show that $C_{\ICOD_0}\nsubseteq C_{\ICOD_{\mathrm{bdd}}}$. For that, let $\mathrm{Id}$ denote the identity on $L^2(\R)$ and  notice that  $\Phi_{\mathrm{Id}}(x)=x$ for all $x\in [0,\infty]$, so    $\mathrm{Id} \in C_{\ICOD_0}$. Fix  $f \in\ICOD_{\mathrm{bdd}}$ and $a \in B_f$. As $f$ is bounded, we must have  $\tau(s_l(a)) < \infty$, so   $s_l(a)$ is a projection strictly below $\mathrm{Id}$. Therefore, if    $\xi\in L^2(\R)$ is a unit vector supported below $\mathrm{Id} - s_l(a)$, we have that 
\begin{align*}
\norm{\mathrm{Id} - a} &\geq \norm{\xi - a\xi} \geq \norm{(\mathrm{Id} - s_l(a))(\xi - a \xi)} = \norm{\xi}
 = 1.
\end{align*} So $d(\mathrm{Id}, B_f) = 1$ and, by the arbitrariness of $f$,  this shows that  $\mathrm{Id}\notin C_{\ICOD_{\mathrm{bdd}}}$.
 
 We now show that $C_{\ICOD_{\mathrm{bdd}}}\nsubseteq C_{\ICOD_{0}}$.
For this, let $(\eta_n)_{n=0}^\infty$ be the standard Haar system, i.e.,   $\eta_0 = \chi_{(0, 1)}$ and then recursively define $\eta_{n+1}(x) = \eta_n(2x) - \eta_n(2x - 1)$. In other words,  for each $n\in \N$, we have 
\[
\eta_n(x)=\left\{\begin{array}{ll}
1,& \text{ if } x \in \bigcup_{k = 0}^{2^{n - 1} - 1} (\frac{2k}{2^n}, \frac{2k + 1}{2^n}),\\
-1, & \text{ if }x\in \bigcup_{k = 0}^{2^{n - 1} - 1} (\frac{2k + 1}{2^n}, \frac{2k + 2}{2^n}),\\
0, &\text{ otherwise}.
\end{array}\right.
\]
So    $(\eta_n)_{n=0}^\infty$ is an  orthonormal basis for $L^2(0,1)$ and  $s(\eta_n) = \chi_{[0, 1]}$ for each $n \geq 0$.

Define an operator $a \in \mathcal{B}(L^2(\R))$ by letting   \[a\xi = \sum_{n = 1}^\infty\sqrt{2^n} \brak{\xi,   \chi_{[\frac{2^n - 2}{2^n}, \frac{2^n - 1}{2^n}]}} \eta_n\ \text{ for all }\ \xi\in L^2(\R).\] 
Since $a\xi\in L^2(0,1)$ for all $\xi\in L^2(\R)$, it follows that $\Phi_a\leq 1$. Also, as 
\[a^* \xi = \sum_{n = 1}^\infty \sqrt{2^n}\brak{\xi,     \eta_n} \chi_{[\frac{2^n - 2}{2^n}, \frac{2^n - 1}{2^n}]}\ \text{ for all }\ \xi\in L^2(\R),\]
it similarly follows that  $\Phi_{a^{*}}\leq 1$  and thus $a \in B_{\ICOD_{\mathrm{bdd}}}$. 

Take $f \in\ICOD_0$, $b \in B_f$, and $\varepsilon > 0$.  Pick $\delta > 0$ and $k\in\N$ such that $f(x) < \varepsilon$ for every $x \in (0, \delta)$ and   $\frac{1}{2^k} < \delta$. Then for  $\xi = \sqrt{2^k}\chi_{[\frac{2^k - 2}{2^k}, \frac{2^k-1}{2^k}]}$ we have 
\begin{align*}
    \norm{a - b}^2 &\geq \norm{(a - b)\xi}^2 = \int_{-\infty}^\infty \Big| \eta_k(x) - (b\xi)(x) \Big|^2 dx \geq 1 - \varepsilon
\end{align*}
again because subtracting $b\xi$ can only change the value of $\eta_k$ on a small set.  By the arbitrariness of $f$, $b$,  and $\eps$, this shows that   $a \notin C_{\ICOD_0}$.
\end{proof}

We now present a dichotomy for families $\cF \subseteq \ICOD_{<\infty}$ containing an element $f$ with $\lim_{x\to 0} f(x) > 0$.

\begin{proposition}\label{propTopPOSETNonZeroLimitAtZero}
Let $\mathcal{F} \subseteq \ICOD_{<\infty}$ be a collection of functions such that $\lim_{x \to 0} f_0(x) > 0$  for some $f_0 \in \mathcal{F}$.  
\begin{enumerate}
    \item\label{ItempropTopPOSETNonZeroLimitAtZero1} If $\lim_{x \to \infty} \frac{f(x)}{x} = 0$ for every $f \in \mathcal{F}$, then $C_{\brak{\mathcal{F}}}= C_{\brak{\mathcal{F}_{\_}}} = C_{\ICOD_{\mathrm{bdd}}}$
    \item\label{ItempropTopPOSETNonZeroLimitAtZero2} If there is $f\in \cF$ so that  $\lim_{x \to \infty} \frac{f(x)}{x} >0 $,  then $C_{\brak{\mathcal{F}}} =C_{\brak{\mathcal{F}_{/}}} = C_{\ICOD_{<\infty}}$. 
\end{enumerate}  
In particular, $C_{\ICOD_{\mathrm{bdd}}}\subseteq C_{\langle \cF\rangle}$.
\end{proposition}
\begin{proof}

\eqref{ItempropTopPOSETNonZeroLimitAtZero1}
By Proposition \ref{propDichotomyAtInfinityFlat}, we have $C_{\brak{\mathcal{F}}} = C_{\brak{\mathcal{F}_{\_}}}$. Hence, as $\brak{\mathcal{F}_{\_}} \subseteq\ICOD_{\mathrm{bdd}}$, it follows   that $C_{\brak{\mathcal{F}}} \subseteq C_{\ICOD_{\mathrm{bdd}}}$. On the other hand, if $g \in\ICOD_{\mathrm{bdd}}$, then there is $n \in \N$ such that $g \leq n   f_0 $. So  $C_{\ICOD_{\mathrm{bdd}}} \subseteq C_{\brak{\mathcal{F}}}$ (Proposition \ref{propDominatedFunctionImpliesSmallerAlgebra}).

\eqref{ItempropTopPOSETNonZeroLimitAtZero2} As $\brak{\mathcal{F}} \subseteq\ICOD_{<\infty}$, it follows that  $C_{\brak{\mathcal{F}}} \subseteq C_{\ICOD_{< \infty}}$. Fix $f\in \cF$ such that $\lim_{x \to \infty} \frac{f(x)}{x} > 0$. Then for any $g \in\ICOD_{< \infty}$ there are $m, n \in \N$ such that $g \leq m   f_0 + n  f$ (cf. proof of Proposition \ref{propDichotomyAtInfinityLinear}). So  $C_{\ICOD_{< \infty}} \subseteq C_{\brak{\mathcal{F}}}$ (Proposition \ref{propDominatedFunctionImpliesSmallerAlgebra}). The equality $C_{\brak{\mathcal{F}}} = C_{\brak{\mathcal{F}_{/}}}$ follows from Proposition \ref{propDichotomyAtInfinityLinear}.
\end{proof}

\begin{theorem}[Theorem \ref{thmTopPOSETContainmentsIntro} in Section \ref{SectionIntro}]
Within $\mathbb{P}$, all the following inclusions are strict and have no intermediate elements.
 \begin{diagram}{1.3em}{0em}
 \matrix[math] (m) { \ & C_{\ICOD_{0}}  & \  & \   \\ C_{\ICOD_{0 \cap \mathrm{bdd}}} & \ & C_{\ICOD_{<\infty}} &  C_{\ICOD} =\cB(L^2(\R))\\ \ &  C_{\ICOD_{\mathrm{bdd}}} & \ & \  \\ }; 
 \path (m-2-1) edge[draw=none] node [sloped, auto=false, allow upside down] {$\subsetneq$} (m-1-2);
 \path (m-1-2) edge[draw=none] node [sloped, auto=false, allow upside down] {$\subsetneq$} (m-2-3);
 \path (m-2-1) edge[draw=none] node [sloped, auto=false, allow upside down] {$\subsetneq$} (m-3-2);
 \path (m-3-2) edge[draw=none] node [sloped, auto=false, allow upside down] {$\subsetneq$} (m-2-3);
 \path (m-2-3) edge[draw=none] node [sloped, auto=false, allow upside down] {$\subsetneq$} (m-2-4);
 \end{diagram} Moreover, $C_{\ICOD_{\mathrm{bdd}}}$ and $C_{\ICOD_0}$ are the only elements  between $C_{\ICOD_{0\cap \mathrm{bdd}}}$ and $C_{\ICOD_{<\infty}}$. \label{thmTopPOSETContainments}
\end{theorem}

\begin{proof}
By Propositions  \ref{propTopPOSETFullNotSubsetOfFinite} and  \ref{propTopPOSETBoundedNotSubsetOfLimitZero}, all the inclusions are strict. So we need to show that there are no intermediate elements in any of the inclusions above, and that $C_{\ICOD_{\mathrm{bdd}}}$ and $C_{\ICOD_0}$ are the only elements between $C_{\ICOD_{0\cap \mathrm{bdd}}}$ and $C_{\ICOD_{<\infty}}$. By Theorem \ref{ThmpropSinglyGeneratedC*AlgCountablyGeneratedByISOD}, we only need to consider families of functions in $\ICOD$. Also, by Proposition \ref{propTopPOSETFullNotSubsetOfFinite}, we already have that $C_{\ICOD}$ is an immediate successor of $C_{\ICOD_{<\infty}}$. 

Next we claim that $C_{\ICOD_{<\infty }}$ is an immediate successor of both $C_{\ICOD_{0 }} $ and $C_{\ISOD_{\mathrm{bdd}}}$.  Let $\mathcal{F}$ be a nonempty family of increasing functions with $C_{\brak{\mathcal{F}}}\subseteq C_{\ICOD_{<\infty }} $. If  $\lim_{x\to 0}f(x)=0$ for all $f\in \cF$, then $\langle \cF\rangle \subseteq \ICOD_{0}$ and we have  $C_{\brak{\mathcal{F}}}\subseteq C_{\ICOD_{0}}$.   If  $\lim_{x\to 0}f_0(x)>0$ for some $f\in \cF$, then  Proposition \ref{propTopPOSETNonZeroLimitAtZero} gives that either $C_{\brak{\mathcal{F}}}= C_{\ICOD_{\mathrm{bdd} }} $ or $C_{\brak{\mathcal{F}}}= C_{\ICOD_{<\infty }} $.  This finishes the claim. 

It remains to show that if $\mathcal{G}$ is a nonempty family of increasing  functions with   $C_{\ICOD_{0 \cap \mathrm{bdd}}} \subsetneq C_{\brak{\mathcal{G}}}$, then at least one of $C_{\ICOD_{0}}$ and $C_{\ICOD_{\mathrm{bdd}}}$ is contained in $C_{\brak \cG}$.

If $\cG\nsubseteq \ICOD_0$,   the inclusion $C_{\ICOD_{ \mathrm{bdd}}} \subseteq C_{\brak{\mathcal{G}}}$  again follows from Proposition \ref{propTopPOSETNonZeroLimitAtZero}.  For the rest of the proof we assume  that $\cG\subseteq \ICOD_0$.

If we were to have $\lim_{x \to \infty} \frac{g(x)}{x} = 0$ for all $g \in \cG$, then    $B_{\langle \cG \rangle} = B_{\langle \cG_{\_} \rangle}$ (Proposition \ref{propDichotomyAtInfinityFlat}).  As   $\cG_{\_} \subseteq \ICOD_{0 \cap \mathrm{bdd}}$, it would follow that $C_{\langle \cG \rangle} = C_{\langle \cG_{
\_} \rangle} \subseteq C_{\ICOD_{0 \cap \mathrm{bdd}}}$, which violates the hypothesis on $C_\cG$.  So there must be $g \in \cG$ with $\lim_{x \to \infty} \frac{g(x)}{x} > 0$.  Choose any $f\in \ICOD_0$. 
Since  $C_{\ICOD_{0 \cap \mathrm{bdd}}} \subseteq C_{\brak{\mathcal{G}}}$, we consult Theorem \ref{thmCharacterizationOfMultiFGenC*AlgsFromFunctionProperties} and conclude that $f_{\_} \in C_{\ICOD_{0 \cap \mathrm{bdd}}}$ fails both conditions of that theorem.  Failing the second means that there is some nonzero $g_0 \in \brak{\mathcal{G}}$ and $m$ such that $f = f_{\_} \leq m g_0$ on some interval $[0,\delta]$ with $\delta \leq 1$.  We may also find large enough $n$ so that $f \leq n g$ on $[\delta, \infty)$.  Putting these together, $f \leq m g_0+ng\in \langle \cG\rangle$. Since $f$ was arbitrary,  $C_{\ICOD_{0}} \subseteq C_{\brak{\mathcal{G}}}$ as required.
\end{proof}

\begin{remark} \label{fullposet}
In the second-named author's PhD dissertation \cite{EisnerDissertation}, the diagram from Theorem \ref{thmTopPOSETContainments} is ``completed" to Figure \ref{FigXY4}.  All information in Figure \ref{FigXY4} is justified in the present paper, except for some diagonal solid lines that rely on results of \cite{EisnerDissertation} concerning the existence of successors in $\mathbb{P}$.

\begin{figure}[h]
\begin{tikzpicture}[scale=1]
  \node (max) at (5 + 1,2) {$\mathcal{B}(L^2(\R))$};
  \node (a) at (2.5 + 1,2) {$C_{\ICOD_{<\infty}}$};
  \node (b) at (2.5,0) {$C_{\ICOD_{bdd}}$};
  \node (c) at (0 + 1,2) {$C_{\ICOD_0}$};
  \node (d) at (0,0) {$C_{\ICOD_{0 \cap bdd}}$};
  
  \node (dots21) at (-30/20 + 1,2) {$C_{\brak{\mathcal{G}}}$};
  \node (dots22) at (-40/20 + 1,2) {};
  \node (dots23) at (-50/20 + 1,2) {};
  \node (dots11) at (-30/20,0) {};
  \node (dots12) at (-40/20,0) {};
  \node (dots13) at (-50/20,0) {$C_{\brak{\mathcal{F}}}$};
  
  \node (g) at (-4 + 1,2) {$C_{\brak{x}}$};
  \node (h) at (-4,0) {$C_{\brak{x_{\_}}}$};
  \node (min) at (-6,0) {$0$};
  \draw[dashed] (g) -- (dots21) -- (c)
  (h) -- (dots13) -- (d);   
  
  \draw (c) -- (a) -- (max)
  (min) -- (h)  (d) -- (b)
  (h) -- (g)
  (dots13) -- (dots23)   (dots12) -- (dots22)   (dots11) -- (dots21)
  (d) -- (c)
  (b) -- (a);
  
  \node (Gdesc1) at (-5.3, 2.2) {\textcolor{red}{$\lim_{x \to \infty} \frac{g(x)}{x} > 0$}};
  \node (Gdesc2) at (-5.45, 1.7) {\textcolor{red}{for some $g \in \mathcal{G}$}};

  \node (Fdesc1) at (5.2, .2) {\textcolor{red}{$\lim_{x \to \infty} \frac{f(x)}{x} = 0$}};
  \node (Fdesc2) at (4.9, -.3) {\textcolor{red}{for all $f \in \mathcal{F}$}};
\end{tikzpicture}
\caption{(Taken from \cite{EisnerDissertation}) Elements farther up and to the right are larger in $\mathbb{P}$. Dotted lines indicate containment. Solid lines indicate immediate successors.\label{FigXY4}} 
\end{figure}
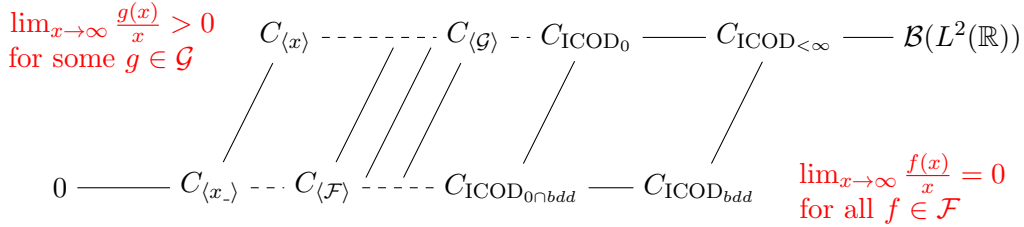 

The two-tiered vertical structure of Figure \ref{FigXY4} reflects the fact that as far as support expansion \cstar-algebras are concerned, the growth of constraints toward $\infty$ comes down to a single question: is there a constraint with linear growth or not?  (The first condition in Theorem \ref{thmCharacterizationOfMultiFGenC*AlgsFromFunctionProperties} guarantees that a ``no'' cannot contain a ``yes''.) The horizontal structure in each tier encodes possible constraint growth at 0 (as sorted by the second condition in Theorem \ref{thmCharacterizationOfMultiFGenC*AlgsFromFunctionProperties}), whose wildness is featured in the next section.

In the discrete setting, one can approach 0 in just three inequivalent ways: as the constant 0, finitely and nonzero, or as the constant $\infty$.  Only the middle option allows a further yes/no refinement regarding the existence of linear growth toward $\infty$.  So the four-element poset of Theorem \ref{thmDiscContrExpC*Algs} may be viewed as a quotient of Figure \ref{FigXY4}, where the middle of Figure \ref{FigXY4} has collapsed to one point for each tier.
\end{remark}

\section{The order structure of some large subsets of $\mathbb P$}\label{SectionLargePoset}

  While Theorem \ref{thmDiscContrExpC*Algs} shows that  $\{0\}$, $\mathcal K(\ell^2(\N))$, $C_{\mathrm{RC}}$, and $\mathcal{B}(\ell^2(\N))$ are the only support expansion \cstar-subalgebras of $\mathcal{B}(\ell^2(\N))$, we prove next that the situation is drastically different in the poset $\mathbb P$ (see Theorem \ref{ThmMAIN}).  By the results of Section \ref{SectionCharFunct}, we have reduced comparability in $\mathbb P$ to function-theoretic questions. We now work primarily with functions and then use Theorem \ref{thmCharacterizationOfMultiFGenC*AlgsFromFunctionProperties} and Corollary \ref{corCharacterizationOfSingleFGenC*AlgsFromFunctionProperties} in order  to find large subsets of $\mathbb P$ with rich order structure.

This entire section is dedicated to the proof of the following:

\begin{theorem}[Theorem \ref{ThmMAINIntro} in Section \ref{SectionIntro}]
${}$
\begin{enumerate}
\item $\mathbb P$ has uncountable ascending chains,
\item $\mathbb P$ has uncountable descending  chains, and
\item $\mathbb P$ has uncountable antichains.
\end{enumerate}\label{ThmMAIN}
\end{theorem}

Several technical results will be needed on the way to Theorem \ref{ThmMAIN}. The following proposition was inspired by a construction in  \cite{Site}.

\begin{proposition}\label{propContPoSetHeightSetup}
  Let $f_0\colon [0, 1] \to  [0, 1]$ be such that  
\begin{enumerate}
    \item\label{ItempropContPoSetHeightSetup1} $f_0$ is increasing,
    \item\label{ItempropContPoSetHeightSetup2} $f_0$ is concave down,
    \item\label{ItempropContPoSetHeightSetup3} $x < f_0(x) < 1$ for $0 < x < 1$,
    \item\label{ItempropContPoSetHeightSetup4} $\lim_{x \to  0} f_0(x) = 0$, and
    \item\label{ItempropContPoSetHeightSetup5} $\lim_{x \to  0} \frac{f_0^{(n)}(x)}{f_0^{(m)}(x)} = \infty$ for all $n > m \geq 0$.
\end{enumerate}
Then for each   countable ordinal $\alpha$ there is  $f_\alpha \colon [0, 1] \to  [0, 1]$ satisfying the same properties above and such that  $\lim_{x \to  0} \frac{f_{\beta}(x)}{f_\alpha^{(n)}(x)} = \infty$ for all $n \in \N$ and all $\beta > \alpha$. 
\end{proposition}

\begin{proof} 
We define $(f_\alpha)_{\alpha<\omega_1}$ by  transfinite induction. Suppose $\beta<\omega_1$ and that  $(f_\alpha)_{\alpha<\beta}$ has been defined. Then, if $\beta$ is a    successor ordinal, say $\beta=\gamma+1$, we define \[f_{\beta}(x)= \sum_{n = 1}^\infty \frac{1}{2^n} f_{\gamma}^{(n)}(x) \ \text{ for all }\ x\in [0,\infty].\]
 If $\beta$ is a   limit ordinal, then its cofinality must be $\omega$, so pick an increasing sequence of ordinals $(\alpha[n])_n $ so that $\alpha=\sup_n\alpha[n]$ and let \[f_\beta(x) = \sum_{n = 1}^\infty \frac{1}{2^n} f_{\alpha[n]}(x)\ \text{ for all }\ x\in [0,\infty].\]

We now show that $(f_\alpha)_{\alpha<\omega_1}$ satisfies the desired properties. Items \eqref{ItempropContPoSetHeightSetup1}, \eqref{ItempropContPoSetHeightSetup2},  \eqref{ItempropContPoSetHeightSetup3}, and \eqref{ItempropContPoSetHeightSetup4} follow straightforwardly. Let us show  $(f_\alpha)_{\alpha<\omega_1}$ also satisfies \eqref{ItempropContPoSetHeightSetup5}. We proceed by transfinite induction. Suppose $\beta<\omega_1$ and that  \eqref{ItempropContPoSetHeightSetup5} holds for all $\alpha<\beta$. By the definition of $f_\beta$, we can pick $\alpha<\beta$ such that $f_\alpha/2\leq f_\beta$. Since $f_\alpha$ satisfies \eqref{ItempropContPoSetHeightSetup5}, we have that  for    any $N \in \N$ there is   $\delta > 0$ such that $\frac{f_\alpha(x)}{x} > 2N$ for all $0 < x < \delta$. Therefore, $\frac{f_{\beta}(x)}{x} > \frac{f_{\alpha}(x)}{2x} > N$ for all $0 < x < \delta$ and, as $N$ is arbitrary,  we have that  $\lim_{x \to  0} \frac{f_{\beta}(x)}{x} = \infty$. Since $\lim_{x \to  0} \frac{f_{\beta}^{(m + 1)}(x)}{f_\beta^{(m)}(x)} = \lim_{x \to  0} \frac{f_\beta(x)}{x}$ for all $m\in\N$ and 
\[ \frac{f_{\beta}^{(n)}(x)}{f_\beta^{(m)}(x)}= \frac{f_{\beta}^{(n)}(x)}{f_\beta^{(n-1)}(x)}\cdot   \ldots  \cdot \frac{f_{\beta}^{(m + 1)}(x)}{f_\beta^{(m)}(x)}\]
  for all $n>m\geq 0$,  \eqref{ItempropContPoSetHeightSetup5} follows.

    Finally, we need to show that $\lim_{x \to  0} \frac{f_{\beta}(x)}{f_\alpha^{(n)}(x)} = \infty$ for all   $ \alpha<\beta<\omega_1$ and all $n \in \N$. We fix $\alpha <\omega_1$ and proceed by induction on $\beta$. Say $\beta=\alpha+1$. Then if $n\in\N$, we have that  $f_\beta > \frac{1}{2^{n+1}} f_\alpha^{(n+1)}$ and thus we have   \[\lim_{x \to  0} \frac{f_\beta(x)}{f_\alpha^{(n)}(x)} \geq \lim_{x \to  0} \frac{f_\alpha^{(n+1)}(x)}{2^{n+1}f_\alpha^{(n)}(x)} = \infty\] by item $(5)$.  Suppose now that  $\alpha+1<\beta$ and fix $n\in\N$. Then, regardless of $\beta$ being a successor or a limit ordinal, there is $\gamma\in (\alpha,\beta)$ such that  $f_\beta > \frac{1}{2^n}f_{\gamma}$. Thus, by the induction hypothesis, we have that \[\lim_{x\to \infty}\frac{f_{\beta}(x)}{f_\alpha^{(n)}(x)}\geq \ \frac{f_{\gamma}(x)}{2^nf_\alpha^{(n)}(x)} = \infty.\] This completes the proof.
\end{proof}

We can now prove the first item of Theorem \ref{ThmMAIN}.

\begin{corollary}\label{corContPoSetHeight}
The partially ordered set $\mathbb P$  has an uncountable ascending chain below $C_{\ICOD_{0\cap \mathrm{bdd}}}$.
\end{corollary}

\begin{proof}
Let $f_0(x) = \sqrt{x} \cdot \chi_{[0, 1]}(x)$ and note that this fulfills all of the hypotheses of Proposition \ref{propContPoSetHeightSetup}. Let $(f_\alpha)_{\alpha<\omega_1}$ be the family given by Proposition \ref{propContPoSetHeightSetup}. By abuse of notation, we extend each $f_\alpha$ to an $\ICOD$ function on the whole $[0,\infty]$ by letting  $f_\alpha(x) = 1$ for any $x > 1$. Notice that   $(f_\alpha)_{\alpha <\omega_1}$ satisfies the following:  $\lim_{x \to  0} f_\alpha(x) = 0$, $\lim_{x \to  \infty} \frac{f_\alpha(x)}{x} = 0$, and  $\lim_{x \to  0} \frac{f_\beta(x)}{f_\alpha^{(N)}(x)} = \infty$ for all   $\alpha< \beta<\omega_1$ and all  $N \in \N$. The containment characterization given by   Corollary \ref{corCharacterizationOfSingleFGenC*AlgsFromFunctionProperties} gives us that both $C_{\langle f_\alpha\rangle}\subseteq C_{\langle f_\beta\rangle}$ and $ C_{\langle f_\beta\rangle}\not\subseteq C_{\langle f_\alpha\rangle}$ for all $\alpha<\beta<\omega_1$. Since $\omega_1$ is uncountable, the result follows.
\end{proof}

Given  $f\in \ISOD$,   we define    $\mathcal{T}f\colon [0, \infty] \to  [0, \infty]$ by letting 
\[
(\mathcal{T}f)(x)=\left\{
\begin{array}{ll}
0, & \text{ if }f(x)=0\\
\frac{x}{f(x)},& \text{ if } x<\infty \text{ and }f(x)\neq 0,\\
\infty, & \text{ if } x=\infty.
\end{array}
\right.
\]
It is straightforward to check that   $f$ also belongs to $\ISOD$.

\begin{proposition}\label{propContPoSetDepthSetup} Suppose $f, g \in\ISOD$ satisfy $\lim_{x \to  0} f(x) =  \lim_{x \to  0} g(x)=0$. Suppose furthermore that  $\lim_{x \to  0} \frac{f(x)}{g^{(N)}(x)} = \infty$ for all $N \in \N$ and $g(x) \geq \sqrt{x}$ for sufficiently small $x$. Then $C_{\langle \mathcal{T}g\rangle} \supsetneq C_{\langle\mathcal{T}f\rangle} \supsetneq C_{\brak{x}}$.
\end{proposition}

\begin{proof}
Since $\lim_{x \to  0} \frac{(\mathcal{T}f)(x)}{x} = \lim_{x \to  0} \frac{1}{f(x)} = \infty$,   Corollary \ref{corCharacterizationOfSingleFGenC*AlgsFromFunctionProperties} implies  that $C_{\brak{\mathcal{T}f}} \nsubseteq C_{\brak{x}}$. On the other hand, as  $\lim_{x \to  0} \frac{x}{(\mathcal{T}f)(x)} = \lim_{x \to  0} f(x) = 0$, Corollary \ref{corCharacterizationOfSingleFGenC*AlgsFromFunctionProperties} implies that   $C_{\brak{x}} \subseteq C_{\brak{\mathcal{T}f}}$.

We are left to show that $C_{\brak{\mathcal{T}g}} \supsetneq C_{\brak{\mathcal{T}f}}$.  We by start showing that $\lim_{x \to  0} \frac{f(x)^n}{g(x)} = \infty$ for all $n\in\N$. For that, fix $n\in\N$ and note that, as $g(x) \geq \sqrt{x}$ for sufficiently small $x$, then   $g^{(n)}(x) \geq \sqrt[2^n]{x}$ for sufficiently small $x$. Therefore, using that $g\leq 1$ on a neighborhood of $0$, we must have that \[\lim_{x \to  0} \frac{f(x)}{\sqrt[n]{g(x)}} \geq \lim_{x \to  0} \frac{f(x)}{\sqrt[2^n]{g(x)}} \geq \lim_{x \to  0} \frac{f(x)}{g^{(n)}(x)} = \infty\] So,   by the continuity of $x^n$, $\lim_{x \to  0} \frac{f(x)^n}{g(x)} = \infty$.
 
  A simple induction gives that  \[(\mathcal{T}f)^{(n)}(x) = \frac{x}{\prod_{k = 0}^{n - 1} f((\mathcal{T}f)^{(k)}(x))}\]  
  for all $x\in [0,\infty]$.   Also, as  $\lim_{x \to  0} f(x) = 0$, we have $(\mathcal{T}f)(x) = \frac{x}{f(x)} > x$ for sufficiently small $x$, which implies that, for each $n \in \N$, $(\mathcal{T}f)^{(n)}(x) > x$ for sufficiently small $x$. Therefore, as $f$ is   increasing,  $\prod_{k = 0}^{n - 1} f((\mathcal{T}f)^{(k)}(x)) \geq f(x)^n$ for sufficiently small $x$. Then, using the result of the previous paragraph, we have that
\begin{align*}
    \lim_{x \to  0} \frac{(\mathcal{T}g)(x)}{(\mathcal{T}f)^{(n)}(x)}  
    = \lim_{x \to  0} \frac{\prod_{k = 0}^{n - 1} f((\mathcal{T}f)^{(k)}(x))}{g(x)} 
    \geq \lim_{x \to  0} \frac{f(x)^n}{g(x)}  
    = \infty .
\end{align*}
 Corollary \ref{corCharacterizationOfSingleFGenC*AlgsFromFunctionProperties} then implies that   $C_{\brak{\mathcal{T}g}} \nsubseteq C_{\brak{\mathcal{T}f}}$. Finally, since   \[\lim_{x \to  0} \frac{(\mathcal{T}f)(x)}{(\mathcal{T}g)(x)} = \lim_{x \to  0} \frac{g(x)}{f(x)} = 0,\]  Corollary \ref{corCharacterizationOfSingleFGenC*AlgsFromFunctionProperties} also gives that  $C_{\brak{\mathcal{T}f}} \subseteq C_{\brak{\mathcal{T}g}}$.
\end{proof}

We can now prove the second item of Theorem \ref{ThmMAIN}.

\begin{corollary}\label{corContPoSetDepth}
The partially ordered set $\mathbb P$ has an uncountable descending chain below $C_{\ICOD_{0}}$.
\end{corollary}

\begin{proof}
Let $f_0(x) = \sqrt{x} \cdot \chi_{[0, 1]}(x)$ and let $(f_\alpha)_{\alpha\in \omega_1}$ be given by  Proposition \ref{propContPoSetHeightSetup}.   As before, we extend each $f_\alpha$ to $[0,\infty]$ by letting  $f_\alpha(x) = 1$ for any $x > 1$. Note that, for any $\alpha<\omega_1$, we have  $\lim_{x \to  0} f_\alpha(x) = 0$ and that $f_\alpha(x) \geq \sqrt{x}$ for sufficiently small $x$. Furthermore, if $\alpha<\beta<\omega_1$, then $\lim_{x \to  0} \frac{f_\beta(x)}{f_\alpha^{(N)}(x)} = \infty$ for every $N \in \N$. By Proposition \ref{propContPoSetDepthSetup}, we have   that $ C_{\brak{\mathcal{T}f_\alpha}} \supsetneq C_{\brak{\mathcal{T}f_\beta}} \supsetneq C_{\brak{x}}$ for all $\alpha<\beta<\omega_1$.  Finally note that each $C_{\brak{\mathcal{T}f_\alpha}} = C_{\brak{(\mathcal{T}f_\alpha)_{**}}}$ is contained in $C_{\ICOD_{0}}$.
\end{proof}

There is only the last item of Theorem \ref{ThmMAIN} left to prove. 

\begin{lemma}\label{lemBeatLinearLimZero} If $f \in\ISOD$ is such that $ C_{\brak{f}} \not\subseteq C_{\brak{x}}$ and $\lim_{x \to  0} f(x) = 0$, then $\lim_{x \to  0} \frac{f(x)}{x} = \infty$. In particular, $\lim_{x \to  0} \frac{f^{(n)}(x)}{f^{(m)}(x)} = \infty$ for $n, m \in \N$ with $n > m$.
\end{lemma}

\begin{proof}
By Corollary \ref{corCharacterizationOfSingleFGenC*AlgsFromFunctionProperties} there is  a   decreasing  sequence  $(x_n)_{n = 1}^\infty$ tending   to $0$ such that $\frac{f(x_n)}{x_n} \geq n$. As $f \in\ISOD$, the function $\frac{f(x)}{x}$ is   decreasing, which  implies  $\lim_{x \to  0} \frac{f(x)}{x} = \infty$.
\end{proof}

\begin{proposition}\label{propNonComparablesForSingleC*Algebra}
If $(f_n)_{n\in\N}\subseteq \ICOD_0$ is such that $   C_{\brak{f_n}} \not\subseteq C_{\brak{x}}$ for all $n\in\N$, then there is $g\in \ICOD_0$ with  $   C_{\brak{g}} \not\subseteq C_{\brak{x}}$ and such that $C_{\brak{f_n}}$ and $C_{\brak{g}}$ are incomparable for all $n\in\N$.
\end{proposition}

\begin{proof}
For didactic reasons, we first prove the proposition with the extra assumption that $(f_n)_{n\in\N}$ is a constant sequence, say  $f=f_n$ for all $n\in\N$.  As $C_{\brak{f}}=C_{\brak{nf}}$ for all $n>0$ (Proposition \ref{propDominatedFunctionImpliesSmallerAlgebra}), and $f$ is concave down, we can assume that $f(1)=1$, which entails that $f(x)\geq x$ for all $x\in [0,1]$.

We now construct the desired function $g$. Our approach will be the following: we    construct $g$ piecewise in such a way that we can use the second item of  Corollary \ref{corCharacterizationOfSingleFGenC*AlgsFromFunctionProperties} in order to guarantee that $C_{\langle g\rangle }$ and $C_{\langle f\rangle }$ are incomparable.

 We start by setting some notation and  pointing  out some elementary facts about certain affine functions and their relation with $f$. Given $x,y,b>0$, we let $\ell[x,y,b]$ be the line   which sends $x$ to $y$ and has $b$ as $y$-intercept, i.e., \[\ell[x,y,b](t)=\frac{y-b}{x}t+b\ \text{  for all }\ t\in \R.\] Since we will only be interested in lines with nonnegative slope, we can assume that $y
\geq b$. The construction of $g$ will be based o n the following: given $x,y,b>0$, \begin{itemize}
 \item as $\lim_{x\to 0}f(x)=0$,  we have that  $\lim_{x\to 0}\frac{\ell[x,y,b](x)}{f^{(n)}(x)}=\infty$ for all $n\in\N$, and \item as  $   C_{\brak{f}} \not\subseteq C_{\brak{x}}$, we have  $\lim_{x\to 0}\frac{f(x)}{x}=\infty$ (see Lemma \ref{lemBeatLinearLimZero}), which in turn implies that $\lim_{x\to 0}\frac{f(x)}{\ell[x,y,0]^{(n)}(x)}=\infty$ for all $n\in\N$.
 \end{itemize} The next two facts isolate the conclusions from those given points which we   need --- we recommend readers to guide themselves from Figure \ref{FigXY2} in the construction of $g$.

\begin{figure}[ht]    
\begin{tikzpicture}[scale=0.43]
    \coordinate (Origin)   at (0,0);
   \coordinate (XAxisMin) at (0,0);
    \coordinate (XAxisMax) at (30,0);
    \coordinate (YAxisMin) at (0,-2);
    \coordinate (YAxisMax) at (0,20);

 \draw[->] (0, 0) -- (27, 0) node[right] { };
  \draw[->] (0, 0) -- (0, 14) node[above] { };
 \draw[  domain=0:13.572088083, smooth, variable=\x, gray ] plot ({\x^3/100}, {\x });

 \draw[  domain=22:25, smooth, variable=\x, black] plot ({\x }, { 13.3214925373});

      \foreach \x in {22}{
      \foreach \y in {0}{
        \node[draw,circle,inner sep=1pt,fill] at (\x,\y) {};}}
 \node[anchor=south] at (22,-1.2) {$x_1$};    
  
  \foreach \x in {0}{
      \foreach \y in {13.3214925373	}{
        \node[draw,circle,inner sep=1pt,fill] at (\x,\y) {};}}
 \node[anchor=south] at (-1.4,12.6214925373	) {$y_1$};     
    
     \draw[  domain=0:22, dotted, variable=\x, gray] plot ({\x }, { 13.3214925373});

     \draw[  domain=0: 13.3214925373, dotted, variable=\x, gray] plot ({22 }, { \x});

 \draw[  domain=5:22, smooth, variable=\x, black] plot ({\x }, {  (9.7	- 8.1910447761194)/(10-5)*(\x-5)+ 8.1910447761194});    
 
    \foreach \x in {5}{
      \foreach \y in {0}{
        \node[draw,circle,inner sep=1pt,fill] at (\x,\y) {};}}
 \node[anchor=south] at (5,-1.2) {$x_2$};    
 
 \foreach \x in {11}{
      \foreach \y in {0}{
        \node[draw,circle,inner sep=1pt,fill] at (\x,\y) {};}}
 \node[anchor=south] at (11,-1.2) {$z_1$};

  \foreach \x in {0}{
      \foreach \y in {8.19104477612	}{
        \node[draw,circle,inner sep=1pt,fill] at (\x,\y) {};}}
 \node[anchor=south] at (-.6,7.49104477612	) {$y_2$};

    \draw[  domain=0:   5	, dotted, variable=\x, gray] plot ({\x }, {8.19104477612 }); 
 
     \draw[  domain=0: 8.19104477612	, dotted, variable=\x, gray] plot ({5 }, { \x});  
     
      \draw[  domain=0: 10.0017910448	, dotted, variable=\x, gray] plot ({11 }, { \x});

 \draw[  domain=0.3:5, smooth, variable=\x, black] plot ({\x }, {  (3.94	-10)/(0.3-7)*(\x-7)+10});
 
  \foreach \x in {0.3}{
      \foreach \y in {0}{
        \node[draw,circle,inner sep=1pt,fill] at (\x,\y) {};}}
 \node[anchor=south] at (0.3,-1.2) {$x_3$};    
 
\foreach \x in {2.2}{
      \foreach \y in {0}{
        \node[draw,circle,inner sep=1pt,fill] at (\x,\y) {};}}
 \node[anchor=south] at (2.2,-1.2) {$z_2$};

  \foreach \x in {0}{
      \foreach \y in {3.94	}{
        \node[draw,circle,inner sep=1pt,fill] at (\x,\y) {};}}
 \node[anchor=south] at (-.6,3.24	) {$y_3$};  
    
      \draw[  domain=0: 3.94	, dotted, variable=\x, gray] plot ({0.3 }, { \x});

        \draw[  domain=0:     5.65850746269	, dotted, variable=\x, gray] plot ({2.2 }, { \x});  
        
                \draw[  domain=0:   0.3	, dotted, variable=\x, gray] plot ({\x }, {3.94 });  
  
 \draw[  domain=0:0.3, dashed, variable=\x, black] plot ({\x }, {  (3.94)/(0.3)*(\x)});

     \end{tikzpicture}     
     \caption{In the graph above, the smooth function represents $f$ and the piecewise linear function represents $g$. Note that the scale is modified so that the general behavior of $g$ with respect to $f$ can be represented in the graph. Also, due to obvious physical restrictions, the graph only depicts the case $n=1$ in Facts \ref{Fact1} and \ref{Fact2}.}\label{FigXY2}
\end{figure}
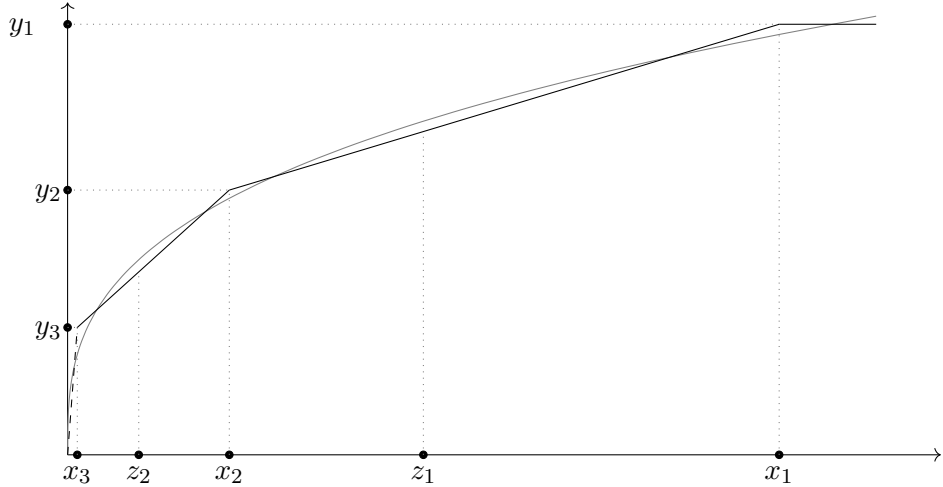

\begin{fact}\label{Fact1}
Given $x,y,b,k,n,z>0$ with $y\geq b$, there is $x'\in (0,z)$   such that $\frac{\ell[x,y,b](x')}{f^{(n)}(x')}>k$.
\end{fact}

\begin{fact}\label{Fact2}
Given  $x,y,b,k,n>0$ with $y\geq b$,   there are $z'\in (0,x)$ and $b'\in (0,b)$ such that $\frac{f(z')}{\ell[x,y,b']^{(n)}(z')}>k$. 
\end{fact}

Let $x_0=y_0=b_0 =1$, $z_0= .9$. Then, alternating between Facts \ref{Fact1} and   \ref{Fact2} (with Fact \ref{Fact1} being the first we use), one can find strictly decreasing  sequences $(x_n)_{n=1}^\infty$,   $(y_n)_{n=1}^\infty$, $(z_n)_{n=1}^\infty$, and $(b_n)_{n=1}^\infty$ in $(0,1]$ tending to $0$ so that 

\begin{enumerate}
\item[(1)] $\frac{l[x_{n-1},y_{n-1},b_{n-1}](x_{n})}{f^{(n)}(x_{n})}>{n}$    for all $n \geq 1$,
\item[(2)] $y_{n}=\ell[x_{n-1},y_{n-1},b_{n-1}](x_{n})$ for all $n \geq 1$, 
\item[(3)] $x_n< z_{n-1}  <x_{n-1}$ for all $n \geq 1$,
\item[(4)] $b_n< y_n$ for all $n \geq 1$,
\item[(5)] $\frac{f(z_{n})}{l[x_{n},y_{n},b_{n}]^{(n)}(z_{n})}>{n}$ for all $n \geq 1$.
\end{enumerate}
Here is how we find the sequences: at stage $n$,  we use Fact \ref{Fact1} to pick $x_n < \min\{z_{n-1}, 2^{-n}\}$ satisfying $(1)$; this determines $(y_n)$ by (2); then use Fact \ref{Fact2} to pick $z_n < x_n$ and $b_n$ satisfying $(4)$ and $(5)$.

We define $g\colon [0,\infty]\to [0,\infty]$ by letting 
\[
g(x)=\left\{\begin{array}{ll}
1,& \text{ if } x\geq x_1,\\
\ell[x_{n},y_{n},b_{n}](x), & \text{ if }x\in (x_{n+1},x_{n}],\\
0,& \text{ if } x=0.
\end{array}\right.
\]
It is clear from its piecewise definition and (2)  that $g$ is     continuous.  Note that the slope of $\ell[x_n,y_n,b_n]$ is  $\frac{y_n-b_n}{x_n}$, while the slope of $\ell[x_{n-1},y_{n-1},b_{n-1}]$ is $\frac{y_n-b_{n-1}}{x_n}$. As $b_n<b_{n-1}$, the slope of $\ell[x_n,y_n,b_n]$ is greater than the slope of $\ell[x_{n-1},y_{n-1},b_{n-1}]$, and both of these are positive by (4).  Thus $g\in \ICOD$.

Let us check that  $  C_{\brak{g}} \not\subseteq C_{\brak{x}} $.  By $(2)$ we have  $g(x_n)=l[x_{n-1},y_{n-1},b_{n-1}](x_n)$ for all $n \geq 1$. Also, it follows from the assumption $x \leq f(x)$ on $[0,1]$ that any iterate of $f$ dominates the identity function, which is its own iterate.  Thus by $(1)$,
 \[ \lim_{n}\frac{g(x_n)}{x_n} =\lim_{n}\frac{l[x_{n-1},y_{n-1},b_{n-1}](x_n)}{x_n} \geq \lim_{n}\frac{l[x_{n-1},y_{n-1},b_{n-1}](x_n)}{f^{(n)}(x_n)} =\infty\]
 and we are done by Corollary \ref{corCharacterizationOfSingleFGenC*AlgsFromFunctionProperties}.

 We now show that $  C_{\brak{g}} \not\subseteq C_{\brak{f}} $.  This again uses Corollary \ref{corCharacterizationOfSingleFGenC*AlgsFromFunctionProperties} and almost the same calculation as in the previous paragraph: 
\[
    \lim_{n\to \infty}\frac{g(x_n)}{f^{(N)}(x_n)} \geq \lim_{n\to \infty}  \frac{g(x_n)}{f^{(n)}(x_n)}   =\lim_{n}\frac{l[x_{n-1},y_{n-1},b_{n-1}](x_n)}{f^{(n)}(x_n)} =\infty .\]

It is left to show $C_{\langle f\rangle}\not\subseteq C_{\langle g\rangle}$.  Since $g(1)= 1$ and $g$ is concave down, we have as before that $g(x)\geq x$ for all $x\in [0,1]$, which gives us that   $g^{(N)}(x)\leq g^{(n)}(x)$ for all $x\in [0,1]$ and all $n\geq N$.  Also, for any $n \geq 1$, $g$ agrees with $l[x_{n},y_{n},b_{n}]$ on an interval and is concave down, so $g\leq  l[x_{n},y_{n},b_{n}]$ on $[0,1]$.  This, plus the fact that all these functions are increasing, entails $g^{(N)}\leq  l[x_{n},y_{n},b_{n}]^{(n)}$ for all $N\in\N$. Therefore
\begin{align*}
    \lim_{n\to \infty}\frac{f(z_n)}{g^{(N)}(z_n)} \geq  \lim_{n\to \infty}\frac{f(z_n)}{g^{(n)}(z_n)} \geq \lim_{n\to \infty}  \frac{f(z_n)}{l[x_{n},y_{n},b_{n}]^{(n)}(z_n)} =\infty .
    \end{align*}
By Corollary \ref{corCharacterizationOfSingleFGenC*AlgsFromFunctionProperties}, this implies that   $C_{\langle f\rangle}\not\subseteq C_{\langle g\rangle}$.  Combined with the conclusion of the previous paragraph, we deduce that $C_{\brak{f}} $ and $ C_{\brak{g}}$ are incomparable.

The result for a single $f$ is now proven, so consider  $(f_n)_n$ as in the statement, i.e.,  $(f_n)_n$ is   not necessarily constant. The proof for this case is actually completely analogous and the only modification needed is that, when using Facts \ref{Fact1} and \ref{Fact2} in order to find strictly decreasing sequence  $(x_n)_{n=1}^\infty$,   $(y_n)_{n=1}^\infty$, $(z_n)_{n=1}^\infty$, and $(b_n)_{n=1}^\infty$ in $[0,1]$ tending to $0$, we must replace items $(1)-(5)$ above by the stronger statements that   $\frac{f_k(z_n)}{l[x_{n},y_{n},b_{n}]^{(n)}(z_{n})} \geq n$ and $\frac{l[x_{n-1},y_{n-1},b_{n-1}](x_n)}{f_k^{(n)}(x_n)} \geq  n$ for all $n\in\N$ and all $k \leq n$. Since this is not an issue, we are done.
\end{proof}

\begin{corollary}\label{corContPoSetWidth}
The partially ordered set $\mathbb P$ has an uncountable antichain below $C_{\ICOD_0}$.
\end{corollary}

\begin{proof}
Let $S$ be the subset of $\mathbb P$ consisting of all $C_{\langle f\rangle}$, with $f\in \ICOD_0$, such that $  C_{\langle f\rangle} \not\subseteq  C_{\langle x\rangle}$.  It follows directly from Corollary \ref{corCharacterizationOfSingleFGenC*AlgsFromFunctionProperties},  that  $C_{\brak{\sqrt{x}}} \in S$, so $S$ is nonempty. Zorn's lemma implies that $S$ contains a maximal antichain, call it $A$. If $A$ is countable, Proposition \ref{propNonComparablesForSingleC*Algebra} gives us   $g \in \ICOD_0$ such that    $  C_{\langle g\rangle} \not\subseteq  C_{\langle x\rangle}$ and    $C_{\brak{g}}$ is pairwise incomparable to everything in $A$. Then $A\cup\{C_{\langle g\rangle}\}$  contradicts the maximality of $A$.
\end{proof}

 \begin{proof}
[Proof of Theorem \ref{ThmMAIN}]
This is simply Corollaries \ref{corContPoSetHeight}, \ref{corContPoSetDepth}, and \ref{corContPoSetWidth}.
 \end{proof}
\begin{acknowledgements}
The author Bruno M. Braga is grateful for the (controlled) support (expansion) from the National Science Foundation under the grant DMS $2054860$. 
\end{acknowledgements}

\newcommand{\etalchar}[1]{$^{#1}$}

\end{document}